\documentclass{amsart}

\usepackage{amsmath, amssymb, amsthm}

\usepackage[colorlinks,linkcolor=blue]{hyperref}
\usepackage{url}

\usepackage[color=green!40,textsize=small]{todonotes}
\usepackage[margin=1.0in]{geometry}               
\usepackage{mathtools}
\usepackage{tikz}
\usetikzlibrary{decorations.markings,shapes.geometric}
\tikzset{->-/.style={decoration={markings, mark=at position #1 with
  {\arrow{>}}},postaction={decorate}}}
\tikzstyle{ball} = [circle,shading=ball, ball color=black,
    minimum size=1mm,inner sep=1.3pt]
\usepackage{tikz-cd}
\tikzset{
      labl/.style={anchor=south, rotate=90, inner sep=.5mm}
      }

\usepackage{enumitem}

\newtheorem{theorem}{Theorem}
\newtheorem{prop}[theorem]{Proposition}
\newtheorem{lemma}[theorem]{Lemma}
\newtheorem{cor}[theorem]{Corollary}

\theoremstyle{definition}
\newtheorem{definition}[theorem]{Definition}
\newtheorem{example}[theorem]{Example}

\newtheorem{remark}[theorem]{Remark}

\newcommand{\overbar}[1]{\mkern 1.5mu\overline{\mkern-1.5mu#1\mkern-1.5mu}\mkern
  1.5mu}

\newcommand{\Z}{\mathbb{Z}}
\newcommand{\Q}{\mathbb{Q}}
\newcommand{\R}{\mathbb{R}}

\newcommand{\T}{\mathbf{T}}
\newcommand{\tL}{\widetilde{L}}

\newcommand{\tH}{\widetilde{H}}
\newcommand{\Hilb}{\mathcal{H}}

\newcommand{\tb}{\tilde{\beta}}

\DeclareMathOperator{\sgn}{\mathrm{sgn}}
\DeclareMathOperator{\rdeg}{\mathrm{rdeg}}
\DeclareMathOperator{\Span}{\mathrm{Span}}
\DeclareMathOperator{\rk}{\mathrm{rank}}

\DeclareMathOperator{\im}{\mathrm{im}}

\DeclareMathOperator{\Hom}{\mathrm{Hom}}

\DeclareMathOperator{\diag}{\mathrm{diag}}

\DeclareMathOperator{\crit}{\mathcal{K}}
\DeclareMathOperator{\Skel}{\mathrm{Skel}}
\DeclareMathOperator{\cok}{\mathrm{cok}}
\DeclareMathOperator{\Div}{\mathrm{Div}}
\DeclareMathOperator{\Pic}{\mathrm{Pic}}
\DeclareMathOperator{\Jac}{\mathrm{Jac}}
\DeclareMathOperator{\class}{\mathcal{J}}

\newcommand{\pup}[1]{\partial_{\,\Upsilon,{#1}}}

\title{Simplicial Dollar Game}
\author{Jesse Kim}
\address{UCSD, La Jolla, CA 92093}
\email{jvkim@ucsd.edu}
\author{David Perkinson}
\address{Reed College, Portland, OR 97202}
\email{davidp@reed.edu}

\subjclass[2010]{05E45}

\begin{document}
\begin{abstract} 
  The dollar game is a chip-firing game introduced by Baker as a context
  in which to formulate and prove the Riemann-Roch theorem for graphs.  A
  divisor on a graph is a formal integer sum of vertices.  Each determines a
  dollar game, the goal of which is to transform the given divisor into one that
  is effective (nonnegative) using chip-firing moves.  We use Duval, Klivans,
  and Martin's theory of chip-firing on
  simplicial complexes to generalize the dollar game and results related to the
  Riemann-Roch theorem for graphs to higher dimensions.  In particular, we
  extend the notion of the degree of a divisor on a graph to a (multi)degree of
  a chain on a simplicial complex and use it to establish two main results.  The
  first of these generalizes the fact that if a
  divisor on a graph has large enough degree (at least as large as the genus of
  the graph), it is winnable; and the second generalizes the fact that trees (graphs of genus~$0$) are exactly the graphs
  on which every divisor of degree~$0$, interpreted as an instance of the dollar
  game, is winnable.
\end{abstract}

\maketitle

\section{Introduction}\label{sect: introduction}
Let~$G=(V,E)$ be a finite, connected, undirected graph with vertex set~$V$ and
edge set~$E$.  To play the dollar game on~$G$, assign an integer number of
dollars to each vertex. Negative integers are interpreted as debt.  A {\em
lending move} consists of a vertex giving one of its dollars to each of its
neighboring vertices, and a {\em borrowing  move} is the opposite, in which a vertex
takes a dollar from each neighbor.  Vertices may lend or borrow,
regardless of the number of dollars they possess.  The goal of the game
is to bring all vertices out of debt through a sequence of such moves.

The dollar game was introduced in {\em Riemann-Roch and Abel-Jacobi theory on a
finite graph}, by Baker and Norine (\cite{Baker}) as a variant of an earlier
version due to Biggs (\cite{Biggs}).  Baker and Norine's work develops the {\em
divisor theory of graphs}, which views a graph as a discrete version of an
algebraic curve or Riemann surface.  The assignment of~$a_v$ dollars to each
vertex~$v$ is formally a~{\em divisor}~$D=\sum_{v\in V}a_vv$ in the free abelian
group~$\Div(G):=\Z V$.  The net amount of money on the graph is
$\deg(D):=\sum_{v\in V}a_v$, the {\em degree} of~$D$.  Divisors~$D$ and~$D'$ are
{\em linearly equivalent}, denoted~$D\sim D'$, if one may be obtained from the
other via lending and borrowing moves.  The group of divisors modulo linear
equivalence is the {\em Picard group}~$\Pic(G)$.  Since lending and borrowing
moves conserve net wealth,~$\Pic(G)$ is graded by degree.  Its degree zero
component is the {\em Jacobian
group}~$\Jac(G)$, which is a finite group with size equal to the number of
spanning trees of~$G$.  A choice of a vertex~$v$ gives an isomorphism
\begin{align}\label{Pic iso}
  \Pic(G)&\xrightarrow{\sim}\Jac(G)\oplus\Z\\\nonumber
  [D]&\mapsto ([D-\deg(D)v],\deg(D)).
\end{align}
A divisor is {\em effective} if its coefficients are nonnegative.  Thus, in the
language of algebraic geometry, an instance of the dollar game is a
divisor~$D\in\Div(G)$, and the game is won by finding a linearly equivalent
effective divisor.

A fundamental concept introduced in~\cite{Baker} is the notion of the~{\em rank}
of a divisor.  If there is no effective divisor linearly equivalent to~$D$, then
the rank of~$D$ is~$r(D)=-1$.  Otherwise, the rank is the maximum
integer~$k$ such that~$D-E$ is linearly equivalent to an effective divisor for
all effective divisors~$E$ of degree~$k$.  In terms of the dollar game, the rank
is a measure of robustness of winnability: the dollar game~$D$ is winnable if
and only if~$r(D)\geq0$, and if~$r(D)=k>0$, it is winnable even after
removing~$k$ dollars arbitrarily.  

The Riemann-Roch theorem for graphs~(\cite[Theorem~1.12]{Baker}) has a form
nearly identical to that for algebraic curves.  It says that for
all~$D\in\Div(G)$,
\[
  r(D)-r(K-D)=\deg(D)+1-g.
\]
Here,~$g=|E|-|V|+1$ and~$K=\sum_{v\in V}\left(\deg_G(v)v-2\right)v$
where~$\deg_G(v)$ is the number of edges incident on~$v$.  These play the role
of the genus and the canonical divisor, respectively, for an algebraic curve.

Since the rank is at least~$-1$, 
\[
  r(D)=\deg(D)+1-g+r(K-D)\geq \deg(D)-g.
\]
A consequence is that if~$\deg(D)\geq g$, then the dollar game~$D$ is winnable.
This result is sharp, too: there are always unwinnable divisors of degree~$g-1$
(\cite[Theorem 1.9]{Baker}).  It follows that all divisors of degree~$0$ are
winnable if and only if~$g=0$, i.e.,~$G$ is a tree.  In summary, the dollar game has
a minimal ``winning degree''~$g$, and that minimal degree is~$0$ exactly when
the game is played on a tree.  Our main goal is to generalize these results to a
dollar game played on a simplicial complex of any dimension.

Lending moves are sometimes called {\em vertex-firings} or {\em chip-firings}
(and borrowing moves are {\em reverse firings}).  They arise naturally as an
encoding of the discrete Laplacian operator for the graph.  Duval, Klivans, and
Martin (\cite{DKM1}, \cite{DKM2}, \cite{DKM3}) use a version of a combinatorial
Laplacian to generalize the divisor theory of graphs to higher-dimensional
simplicial (and cellular) complexes.  In this theory, an~$i$-chain---a
formal integer sum of~$i$-dimensional faces---of a complex~$\Delta$ may be
thought of as an assignment of an integer ``flow'' to each~$i$-face.  Firing
an~$i$-face~$f$ then diverts flow around the~$(i+1)$-faces incident on~$f$.  The
group of~$i$-cycles modulo these firing moves is the {\em $i$-th critical group}
of the complex,~$\crit_i(\Delta)$, generalizing the Jacobian group of a graph.
By \cite[Corollary 4.2]{DKM1}, under certain restrictions on~$\Delta$, the size of
the torsion part of~$\crit_i(\Delta)$ is the number of torsion-weighted
$(i+1)$-dimensional spanning trees of~$\Delta$.

In this paper, we interpret Duval, Klivans, and Martin's theory as a
higher-dimensional dollar game.  A chain on a simplicial complex is thought of
as a distribution of wealth among the faces.  The goal of the game is to use
face-firings to redistribute wealth, leaving no face in debt.  For this purpose,
the naive version of degree as the net wealth of the system is not appropriate:
using that notion of degree, there would be simplicial complexes with chains of
arbitrarily negative degree that are winnable and arbitrarily positive degree
that are unwinnable.  The root of the problem is that, unlike for graphs,
lending and borrowing moves on simplicial complexes are not necessarily
conservative.  Instead, in Definition~\ref{def: degree} we introduce a natural
generalization of the degree of a divisor on a graph to one that is invariant
under firing moves on the chains of a complex.  Our main results generalize the
properties of divisors on graphs discussed in connection with the Riemann-Roch
theorem, above: Theorem~\ref{thm: main} shows that if the degree of a chain is
sufficiently large, then it is winnable, and Corollary~\ref{cor: t=1} shows that
for each~$i$,  all~$(i-1)$-chains of degree~$0$ are winnable if and only if
the~$i$-skeleton of the complex is a spanning forest, torsion-free in
codimension one.

Section~\ref{sect: preliminaries} sets notation and presents required background
on (abstract) simplicial complexes and polyhedral cones. In particular,~$\Delta$
always denotes a~$d$-dimensional simplicial complex. In Section~\ref{sect: the
dollar game}, we recall the definition of the~$i$-dimensional Laplacian~$L_i$
and critical group~$\crit_i(\Delta)$ for~$\Delta$ and use these to
carefully define the dollar game determined by each~$i$-chain.  Two~$i$-chains
are {\em linearly equivalent} if their difference is in the image of~$L_i$.

Section~\ref{sect: degree} defines the degree of each~$i$-chain~$\sigma$
of~$\Delta$ and relates it the winnability of the dollar game, generalizing
results from graphs (the special case~$d=1$) to higher dimensions.  Let~$\Hilb$
be the minimal additive basis, i.e., the Hilbert basis, for the monoid of
nonnegative integer points in the kernel of~$L_i$.  Using~$\Hilb$, we define the
degree of~$\sigma$ as an integer vector~$\deg(\sigma)\in\Z^{|\Hilb|}$.
By~Proposition~\ref{prop: invariance}, the degree of a chain is invariant under
linear equivalence, with the immediate consequence (Corollary~\ref{cor: winnable
implies nonnegative}) that if the dollar game determined by the chain~$\sigma$
is winnable, then~$\deg(\sigma)\geq0$.  Lemma~\ref{lemma: positive element} is a
key technical result showing there is a strictly positive element in the kernel
of~$L_i$.  By Theorem~\ref{thm: deg 0 iso}, the group of degree zero~$i$-chains
modulo linear equivalence is isomorphic to the torsion part of the~$i$-th
critical group.  In the special case where~$d=1$, this result generalizes the
fact that the Jacobian group of a connected graph is the torsion part of the Picard
group (in accordance with isomorphism~\eqref{Pic iso}).  Theorem~\ref{thm: main}
achieves one of our main goals: it says that if the degree of a chain is
sufficiently large, its corresponding dollar game is winnable.  

Section~\ref{sect: pseudomanifolds} considers the case
where~$\Delta$ is a pseudomanifold.  We compute the critical group of an
oriented pseudomanifold (Proposition~\ref{prop: pseudomanifold critical group}),
generalizing~\cite[Theorem~4.7 and subsequent remarks]{DKM1}.  Our main result
on pseudomanifolds is a combinatorial description of the Hilbert basis~$\Hilb$,
described above, in codimension one (Theorem~\ref{thm: pseudomanifold Hilbert
basis}).  The section ends with an example of calculating minimal
degrees~$\delta$ such that every chain of degree at least as large as~$\delta$ is
winnable.

Section~\ref{sect: forests} builds on the work of Duval, Klivans, and Martin
(\cite{DKM1}, \cite{DKM2}, \cite{DKM3}) on higher-dimensional forests and
critical groups.  Our main result is Corollary~\ref{cor: t=1},
which shows that all~$(i-1)$-chains of degree zero are winnable if and only if
the~$i$-skeleton is an~$i$-dimensional spanning forest, torsion-free in
codimension one. We also generalize Theorem~3.4 of~\cite{DKM1}, which for each
dimension gives an isomorphism between the critical group and the cokernel of
the {\em reduced Laplacian}---a submatrix of the Laplacian determined by a
spanning forest. In Section~\ref{sect: staco}, we consider an alternative
generalization of the set of divisors of nonnegative degree on a graph due
to Corry and Keenan~(\cite{CorryKeenan}).  We use it to characterize
higher-dimensional spanning trees that are acyclic in codimension one in terms
of winnability of the dollar game.

Section~\ref{sect: further work} poses some open questions.  Finally,
the proofs of Proposition~\ref{prop: pseudomanifold critical group} and
Theorem~\ref{thm: reduced Laplacian iso} are relegated to an appendix to avoid
distraction from our main line of argument.   

Readers interested in learning more about chip-firing on graphs and its relation to a
diverse range of mathematics may wish to consult the textbooks~\cite{Corry}
and~\cite{Klivans}.

\subsection*{Acknowledgments} 
We would like to thank Scott Corry for the idea of thinking of chip-firing on
simplicial complexes in terms of the dollar game and for sharing some of his
unpublished joint work with Liam Keenan, motivating the results in
Section~\ref{sect: staco}.  We thank Collin Perkinson and our anonymous referee for comments on the exposition.

\section{Preliminaries}\label{sect: preliminaries}
\subsection{Simplicial complexes}\label{sect: simplicial complexes} Throughout this paper,~$\Delta$ is
a~$d$-dimensional simplicial complex on the set $V=[n]:=\left\{ 1,\dots,n
\right\}$ for some integer~$n$.  A subset of~$V$ of cardinality~$i+1$ that is an
element of~$\Delta$ is an~{\em
$i$-dimensional face} or~{\em $i$-face} of~$\Delta$, and the collection of
all~$i$-faces is denoted~$\Delta_i$.  Let~$f_i=f_i(\Delta):=|\Delta_i|$ be the
number of faces of dimension~$i$.  The empty set is the single face of
dimension~$-1$.  The elements of~$V$ are called {\em vertices}.  The set of all
faces forms a poset under inclusion, graded by dimension, and its maximal
elements are the {\em facets} of~$\Delta$.  To say that~$\Delta$ has
dimension~$d$ means that its highest-dimensional facet has dimension~$d$.  The
complex~$\Delta$ is {\em pure} if all of its facets have dimension~$d$, which we
do not assume.  If~$R$ is a commutative ring, the module of~{\em
$i$-chains},~$C_i(\Delta,R)$, is the free~$R$-module with basis~$\Delta_i$.  In
particular, let~$C_i(\Delta)$ denote the integral~$i$-chains, $C_i(\Delta,\Z)$.
Take~$C_i(\Delta,R)=0$ for~$i>d$ and~$i<-1$, whereas~$C_{-1}(\Delta,R)\approx
R$.  Given an~$i$-chain~$\sigma=\sum_{f\in\Delta_i}a_ff$, we
write~$\sigma(f):=a_f$ and define the {\em support} of~$\sigma$ to be
$\mathrm{supp}(\sigma):=\left\{ f\in\Delta_i:\sigma(f)\neq0 \right\}$.

In general, our results will depend on the choice of an orientation
of~$\Delta$ (cf.~Example~\ref{example: orientation}). In order for the dollar game to be sensible, this orientation must be acyclic, i.e., for all $i$, every positive sum of $i$-faces has nonzero boundary.  Since any such orientation induces an acyclic orientation on the 1-skeleton of~$\Delta$, every acyclic orientation is the standard
orientation up to renumbering of the vertices,  so we fix the standard
orientation on~$\Delta$ induced by the natural ordering on the vertex
set~$V=[n]$.  Thus, each~$i$-face is represented by the list of its
vertices~$\overbar{v_0\cdots v_i}$ with~$v_0<\cdots< v_i$.  We fix the
lexicographic total ordering on each~$\Delta_i$ and the corresponding induced
isomorphism~$C_i(\Delta)\simeq\Z^{f_i}$.  If~$\pi$ is a permutation, we write 
\[
  \overbar{v_{\pi(0)}\cdots v_{\pi(i)}}=\sgn(\pi)\,\overbar{v_0\cdots v_i}
\] 
as chains. 

For each~$i$, there is a boundary mapping
\[
\partial_i\colon C_i(\Delta,R)\to C_{i-1}(\Delta,R)  
\] 
defined by 
\[
\partial_i(\overbar{v_0\cdots
v_i}):=\sum_{j=0}^i(-1)^{j}\overbar{v_0\cdots\widehat{v_j}\cdots v_i},
\] 
where~$\widehat{v_j}$ indicates that~$v_j$ is omitted. We
have~$\partial_i\circ\partial_{i+1}=0$.  The elements
of~$\ker\partial_i$ are the~{\em $i$-cycles} and elements of~$\im\partial_i$ are
{\em $i$-boundaries}.  The~$i$-th 
{\em reduced homology} group is 
\[ 
  \tH_i(\Delta,R):=\ker\partial_i/\im\partial_{i+1}.
\]
 The ordinary homology groups~$H_i(\Delta,R)$ use the same definition, with one change:~$\partial_0$ is
taken to be the zero mapping, or equivalently,~$C_{-1}(\Delta)$ is defined to be
the trivial group.  We write simply~$\tH_i(\Delta)$ and~$H_i(\Delta)$ in the
  case~$R=\Z$.  The~$i$-th {\em reduced
Betti number} is
\[
  \tb_i(\Delta)=\rk_{\Z}\tH_i(\Delta)=\dim_{\Q}\tH_i(\Delta,\Q).
\]
Applying the functor~$\Hom(\ \cdot\ ,R)$, we get the dual mapping 
\[ 
  \partial_{i+1}^t\colon C_i(\Delta,R)\to C_{i+1}(\Delta,R) 
\] 
identifying chain modules with their duals using our fixed orderings of the
faces of~$\Delta$.

If~$\Sigma$ is a subcomplex of~$\Delta$, we assume it has the orientation
inherited from~$\Delta$ (induced by the natural ordering on~$V$) and may
write~$\partial_{\Sigma,i}$ for its~$i$-th boundary mapping.  The {\em
$i$-skeleton} of~$\Delta$, denoted~$\Skel_i(\Delta)$, is the subcomplex
consisting of all faces of~$\Delta$ of dimension~$i$ or less.  

Relative homology is mentioned in Section~\ref{sect: pseudomanifolds}.
The {\em relative chain complex} (with~$\Z$-coefficients) for a nonempty
subcomplex~$\Sigma$ of~$\Delta$ is the complex
\[
  \cdots\to C_i(\Delta)/C_i(\Sigma)\xrightarrow{\overbar{\partial}_i}
  C_{i-1}(\Delta)/C_{i-1}(\Sigma)\to\cdots,
\]
where~$\overbar{\partial}_i$ is induced by~$\partial_i$.  The~{\em $i$-th relative
  homology group} is
  \[
    H_i(\Delta,\Sigma):=\ker\overbar{\partial}_i/\im\overbar{\partial}_{i+1}.
  \]
If~$\Sigma=\emptyset$, we take~$H_i(\Delta,\Sigma):=H_i(\Delta)$.

\subsection{Polyhedral cones} We recall some facts about polyhedral cones,
using~\cite{Fulton},~\cite{Henk}, and~\cite{Schrijver} as references.  Let~$Q$
be a {\em cone} in~$\R^n$.  For us, this means~$Q$ is a subset of~$\R^n$ closed
under nonnegative linear combinations: if~$x,y\in Q$
and~$\alpha,\beta\in\R_{\geq0}$, then~$\alpha x+\beta y\in Q$.  The cone~$Q$ is
{\em pointed} if~$Q\setminus\left\{ 0 \right\}$ is contained in an open
half-space in~$\R^n$, i.e., there exists~$z\in\R^n$ such that~$x\cdot z>0$ for
all~$x\in Q\setminus\left\{ 0 \right\}$ (using the ordinary dot product
on~$\R^n$).  We say~$Q$ is {\em polyhedral} if it is finitely generated, i.e.,
if there exist~$x_1,\dots,x_{\ell}\in\R^n$ such that
\[
  Q = \Span_{\R_{\geq0}}\left\{ x_1,\dots,x_{\ell} \right\}
  :=\left\{\textstyle \sum_{i=1}^{\ell}\alpha_ix_i:\alpha_i\geq0 \text{ for
  $1\leq i\leq\ell$}\right\}.
\]
If the generators~$x_1,\dots,x_{\ell}$ can be taken to be integral,
then~$Q$ is a {\em rational} polyhedral cone.  

Let~$Q$ be a rational polyhedral cone.  Then the semigroup of its integral
points,~$Q_{\Z}:=Q\cap\Z^n$, has a {\em Hilbert basis}~$\Hilb$, defined to be a set
of minimal cardinality such that every point of~$Q_{\Z}$ is a nonnegative integral
combination of elements of~$\Hilb$. If~$Q$ is
pointed, then~$\Hilb$ is unique, determined by the property that~$x\in \Hilb$ if and
only if~$x\in Q_{\Z}\setminus\left\{ 0 \right\}$ and there do not exist~$y,z\in
Q_{\Z}\setminus\left\{ 0 \right\}$ such that~$x=y+z$.  If~$Q$ is integrally generated
by $x_1,\dots,x_{\ell}$, let
\[
  \Pi :=
  \Pi(x_1,\dots,x_{\ell}):=\left\{\textstyle\sum_{i=1}^{\ell}\alpha_ix_i:0\leq\alpha_i<1
  \text{ for~$1\leq i\leq\ell$}\right\}\subset\R^{n}
\]
be the corresponding {\em fundamental parallelepiped}.  Then
\[
  \Hilb\subset \left\{ x_1,\dots,x_{\ell} \right\}\cup\Pi.
\]

The {\em dual} of~$Q$ is the rational polyhedral cone
\[
  Q^*:=\{x\in\R^n: x\cdot q\geq 0 \text{ for all $q\in Q$}\},
\]
and we have ~$(Q^*)^*=Q$.  The {\em Minkowski sum} of two rational polyhedral cones~$Q_1$ and~$Q_2$
is the rational polyhedral cone $Q_1+Q_2:=\left\{ x+y:x\in Q_1, y\in
Q_2\right\}$.
We will need the following well-known fact:
\[
  (Q_1\cap Q_2)^*= Q_1^*+Q_2^*.
\]

\subsection{Partial order}\label{sect: partial order}
Throughout this paper, fix the following ``component-wise'' partial order on
the~$i$-chains of~$\Delta$: write~$\sigma\geq\tau$ if~$\sigma(f)\geq\tau(f)$ for
all faces~$f\in\Delta_i$.  We say~$\sigma$ is {\em nonnegative}
if~$\sigma\geq0$, where~$0$ denotes the zero~$i$-chain.  Fix a similar partial
order on~$\R^k$: write~$v\geq w$ if~$v_i\geq w_i$ for all~$i$; and~$v$ is
nonnegative if~$v\geq0$, where~$0$ denotes the zero vector.

\section{The dollar game}\label{sect: the dollar game}
The~{\em
$i$-th Laplacian} of~$\Delta$, also know as the~$i$-th {\em up-down
combinatorial Laplacian}, is the mapping
  \[
    L_i:=\partial_{i+1}\circ\partial_{i+1}^t\colon C_i(\Delta)\to
    C_i(\Delta).
  \]
The isomorphism~$C_i(\Delta)\simeq\R^{f_i}$ identifies~$L_i$ with an~$f_i\times
f_i$ matrix whose rows and columns are indexed by the~$i$-faces.  

Think of~$\sigma=\sum_{f\in\Delta_i}\sigma(f)f\in C_i(\Delta)$
as a distribution of wealth to the~$i$-faces of~$\Delta$:  face~$f$
has~$\sigma(f)$ dollars, interpreted as debt if~$\sigma(f)$ is negative.  A {\em
borrowing move} at an~$i$-face~$f$ redistributes wealth by replacing~$\sigma$ by
the~$i$-chain
\[
 \sigma+L_if.
\]
A {\em lending move} at~$f$ replaces~$\sigma$ by
\[
\sigma-L_if.
\]
The goal of the {\em dollar game} for~$\sigma$ is to bring all faces out
of debt through a sequence of lending and borrowing moves.  In detail,
say~$\sigma$ is {\em linearly equivalent} to the~$i$-chain~$\sigma'$ and
write~$\sigma\sim\sigma'$ if there exists~$v\in\Z^{f_i}$ such that
\begin{equation}\label{eqn: linear equivalence}
  \sigma'=\sigma+L_iv.
\end{equation}
Call~$\sigma'$ {\em effective} if~$\sigma'\geq0$.  Then~$\sigma$ is {\em
winnable} if there exists an effective~$\sigma'$ linearly equivalent
to~$\sigma$, and {\em winning} the dollar game determined by~$\sigma$ means
finding such a~$\sigma'$.

The {\em $i$-chain class group} is
\[
  \class_i(\Delta):=C_i(\Delta)/\!\!\sim\ =\ C_i(\Delta)/\im L_i.
\]
So an~$i$-chain~$\sigma$ is winnable if and only if there is an effective chain in its
class~$[\sigma]\in\class_i(\Delta)$.
  
The image of the~$i$-th Laplacian is contained in the kernel of the~$i$-th boundary
mapping, which allows us to define the {\em $i$-th critical group} of~$\Delta$
introduced by Duval, Klivans, and Martin in~\cite{DKM1}:
\[
  \crit_i(\Delta):=\ker\partial_i/\im L_i.
\]
Choosing a splitting~$\rho\colon \im\partial_i\to C_i(\Delta)$ of the exact
sequence of free abelian groups
\[
  0\to\ker\partial_{i}\to C_i(\Delta)\to\im\partial_i\to0
\]
gives a corresponding isomorphism
\begin{align}\label{J iso}
  \class_i(\Delta)&\to \crit_i(\Delta)\oplus\im\partial_i\\\nonumber
  [\sigma]&\mapsto ([\sigma-\rho(\sigma)],\partial_i(\sigma)).
\end{align}
The torsion part of~$\class_i(\Delta)$ is thus the torsion part of the
critical group,~$\T(\crit_i(\Delta))$, (which, itself, is sometimes called the
  critical group of~$\Delta$ (e.g., in~\cite{DKM2})).
  There is a natural surjection~$\crit_i(\Delta)\to\tH_i(\Delta)$ which is an
  isomorphism when restricted to the free parts of each group
  (Corollary~\ref{cor: critical group and homology}).  
\medskip

\begin{figure}[htb] 
  \centering
  \begin{tikzpicture}[scale=0.9] 
    \def\x{6}
    \begin{scope}
      \node[draw,circle,inner sep=0.5mm] (1) at (0,-1.732) {{\small $1$}};
      \node[draw,circle,inner sep=0.5mm] (2) at (-1,0) {{\small $2$}};
      \node[draw,circle,inner sep=0.5mm] (3) at (1,0) {{\small $3$}};
      \node[draw,circle,inner sep=0.5mm] (4) at (0,1.732) {{\small $4$}};
      \draw[->-=0.55] (1)--node [left=2mm]{$-\$1$} (2);
      \draw[->-=0.55] (1)--node [right=2mm]{$\$2$} (3);
      \draw[->-=0.55] (2)--node [above=1mm]{$-\$3$} (3);
      \draw[->-=0.55] (2)--node [above left=1mm]{$\$2$} (4);
      \draw[->-=0.55] (3)--node [above right=1mm]{$-\$1$} (4);
      \draw[->] (2,0)--node [above=0.8mm]{$\overbar{13}$} (4,0);
      \draw[->] (2,0)--node [below=0.8mm]{lends} (4,0);
    \end{scope}
    \begin{scope}[xshift=\x cm]
      \node[draw,circle,inner sep=0.5mm] (1) at (0,-1.732) {{\small $1$}};
      \node[draw,circle,inner sep=0.5mm] (2) at (-1,0) {{\small $2$}};
      \node[draw,circle,inner sep=0.5mm] (3) at (1,0) {{\small $3$}};
      \node[draw,circle,inner sep=0.5mm] (4) at (0,1.732) {{\small $4$}};
      \draw[->-=0.55] (1)--node [left=2mm]{$\$0$} (2);
      \draw[->-=0.55] (1)--node [right=2mm]{$\$1$} (3);
      \draw[->-=0.55] (2)--node [above=1mm]{$-\$2$} (3);
      \draw[->-=0.55] (2)--node [above left=1mm]{$\$2$} (4);
      \draw[->-=0.55] (3)--node [above right=0.8mm]{$-\$1$} (4);
      \draw[->] (2,0)--node [above=0.8mm]{$\overbar{23}$} (4,0);
      \draw[->] (2,0)--node [below=0.8mm]{borrows} (4,0);
    \end{scope}
    \begin{scope}[xshift=2*\x cm]
      \node[draw,circle,inner sep=0.5mm] (1) at (0,-1.732) {{\small $1$}};
      \node[draw,circle,inner sep=0.5mm] (2) at (-1,0) {{\small $2$}};
      \node[draw,circle,inner sep=0.5mm] (3) at (1,0) {{\small $3$}};
      \node[draw,circle,inner sep=0.5mm] (4) at (0,1.732) {{\small $4$}};
      \draw[->-=0.55] (1)--node [left=2mm]{$\$1$} (2);
      \draw[->-=0.55] (1)--node [right=2mm]{$\$0$} (3);
      \draw[->-=0.55] (2)--node [above=1mm]{$\$0$} (3);
      \draw[->-=0.55] (2)--node [above left=1mm]{$\$1$} (4);
      \draw[->-=0.55] (3)--node [above right=0.8mm]{$\$0$} (4);
    \end{scope}
  \end{tikzpicture} 
  \caption{Winning the dollar
  game~$\sigma=-\overbar{12}+2\cdot\overbar{13}-3\cdot\overbar{23}+2\cdot\overbar{24}-\overbar{34}$
on the~$2$-dimensional simplicial complex with facets~$\overbar{123}$
and~$\overbar{234}$.}\label{fig: diamond complex} 
\end{figure}
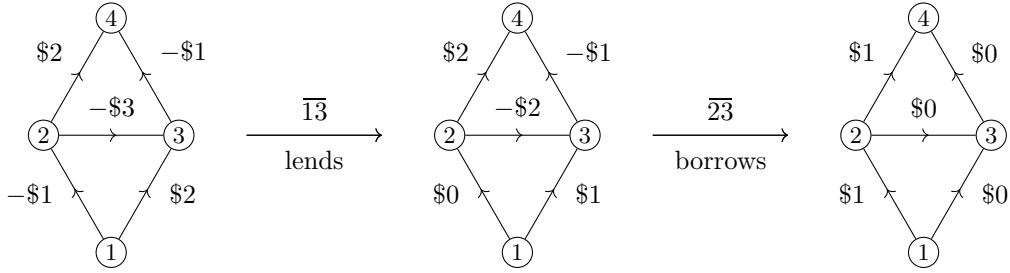

\begin{example}\label{example: simplicial dollar game} Figure~\ref{fig: diamond
  complex} illustrates an instance of the dollar game determined by
  a~$1$-chain~$\sigma$ on the simplicial complex with two
  facets:~$\overbar{123}$ and~$\overbar{234}$.  Calling the winning chain on the
  right~$\sigma'$, Equation~\eqref{eqn: linear
  equivalence} in this case takes the form
 \begin{center}
   \begin{tikzpicture}
     \def\a{-0.7}
     \def\d{0.8}
     \node at (0,0) {$\displaystyle
       \left[\begin{array}{c}
	   1\\0\\0\\1\\0
       \end{array} \right]
       =
       \left[\begin{array}{r}
	   -1\\2\\-3\\2\\-1
       \end{array} \right]
       +
      \left[\begin{array}{rrrrr}
	  1 & -1 & 1 & 0 & 0 \\
	  -1 & 1 & -1 & 0 & 0 \\
	  1 & -1 & 2 & -1 & 1 \\
	  0 & 0 & -1 & 1 & -1 \\
	  0 & 0 & 1 & -1 & 1
      \end{array}\right]
      \left[\begin{array}{r}
	 0\\-1\\1\\0\\0 
     \end{array} \right]
     $};
     \node at (\a,1.4) {$\overbar{12}$};
     \node at (\a+\d,1.4) {$\overbar{13}$};
     \node at (\a+2*\d,1.4) {$\overbar{23}$};
     \node at (\a+3*\d,1.4) {$\overbar{24}$};
     \node at (\a+4*\d,1.4) {$\overbar{34}$};
     \node at (0,-1.8) {$\sigma'=\sigma+L_1 v$.};
   \end{tikzpicture}
 \end{center}
 Note that in moving from~$\sigma$ to~$\sigma'$, money has been introduced from
 nowhere: the net amount in~$\sigma$ is~$-\$1$, while in~$\sigma'$ it is~$\$2$.
 While the simplicial dollar game does not conserve the net amount of money,
 other quantities are conserved, and we will discuss this at length starting in
 the next section.  For now, as an example, it is easy to check that the sum of
 the amount of money on just the edges~$\overbar{12}$ and~$\overbar{13}$ is
 conserved under lending and borrowing moves.  Thus, for instance, if we change
 the amount of money on~$\overbar{12}$ in~$\sigma$ from~$-\$1$ to~$-\$3$, the
 resulting game could never be won.  And that statement would continue to hold
 no matter how much money we added to the edges ~$\overbar{23}$,~$\overbar{24}$,
 and~$\overbar{34}$.
\end{example}

\begin{example}\label{example: orientation}  Here we show that winnability depends on the orientation
  of the simplicial complex.  Figure~\ref{fig: orientation matters} depicts two
  dollar games on the~$2$-simplex (the simplicial complex with the single
  facet~$\overbar{123}$).  The first can be won by lending at the
  edge~$\overbar{13}$.  The second is not winnable.  To see this, note that the
  sum of the~$\overbar{13}$ and~$\overbar{23}$ components of a~$1$-chain on
  this complex---which is~$-\$2$ for the second game---is invariant under lending
  and borrowing moves. So one of these games is winnable and the other is
  not, yet they are the same up to a relabeling of the vertices (which amounts
  to a change in orientation). 
\begin{figure}[htb] 
  \centering
  \begin{tikzpicture}[scale=0.9] 
    \begin{scope}
      \node[draw,circle,inner sep=0.5mm] (1) at (0,0) {{\small $1$}};
      \node[draw,circle,inner sep=0.5mm] (2) at (-1,1.732) {{\small $2$}};
      \node[draw,circle,inner sep=0.5mm] (3) at (1,1.732) {{\small $3$}};
      \draw[->-=0.55] (1)--node [left=2mm]{$-\$1$} (2);
      \draw[->-=0.55] (1)--node [right=2mm]{$\$1$} (3);
      \draw[->-=0.55] (2)--node [above=1mm]{$-\$1$} (3);
    \end{scope}
    \begin{scope}[xshift=6 cm]
      \node[draw,circle,inner sep=0.5mm] (1) at (0,0) {{\small $1$}};
      \node[draw,circle,inner sep=0.5mm] (2) at (-1,1.732) {{\small $2$}};
      \node[draw,circle,inner sep=0.5mm] (3) at (1,1.732) {{\small $3$}};
      \draw[->-=0.55] (1)--node [left=2mm]{$\$1$} (2);
      \draw[->-=0.55] (1)--node [right=2mm]{$-\$1$} (3);
      \draw[->-=0.55] (2)--node [above=1mm]{$-\$1$} (3);
    \end{scope}
  \end{tikzpicture} 
  \caption{Two dollar games on the edges of a~$2$-simplex.  Only the first is
  winnable.}\label{fig: orientation matters} 
\end{figure}
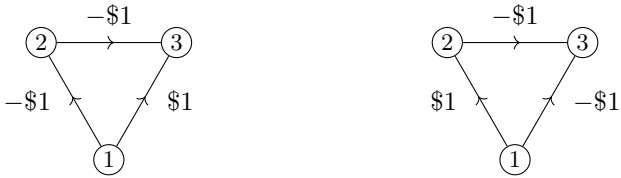
\end{example}

\begin{example}[Graphs] Let~$\Delta=G$ be a connected, undirected graph as in
  the introduction.  In that case, the dollar game for~$0$-chains on~$\Delta$ we
  just defined is the same as the dollar game for graphs from~\cite{Baker}.  If
  the vertices of~$G$ are~$v_i=i$ for~$i=1,\dots,n$, then the~$0$-th Laplacian
  is the usual discrete Laplacian for a graph:
  \[
    L_0 = \diag(\deg_G(v_1),\dots,\deg_G(v_n)) - A,
  \]
  the difference of the diagonal matrix of vertex degrees and the adjacency matrix of~$G$.
  The~$0$-chain class group and~$0$-th critical group are the Picard group and
  Jacobian group, respectively, described in the
  introduction:~$\class_0(\Delta)=\Pic(G)$ and~$\crit_0(\Delta)=\Jac(G)$.
  Isomorphism~(\ref{J iso}) specializes to the usual isomorphism~(\ref{Pic iso})
  for graphs.
\end{example}

\section{Degree}\label{sect: degree}
The naive way of generalizing the degree of a divisor on a graph to the degree
of an~$i$-chain on a simplicial complex~$\Delta$, by simply summing up the
coefficients of the $i$-faces, fails to retain many of the useful properties of
the graph-theoretic degree. Under this naive definition of degree, as shown in
Example~\ref{example: simplicial dollar game}, linearly equivalent~$i$-chains
can fail to have the same degree,~$i$-chains with negative degree can be
winnable, and for a fixed complex, there can exist~$i$-chains of arbitrarily
large degree that are unwinnable.  This section will introduce a better
generalization of degree, avoiding these problems. To summarize the rest of this
section: Theorem~\ref{thm: deg 0 iso} shows that the group of~$i$-chains of
degree zero modulo firing rules is exactly the torsion part of the~$i$-th
critical group, as it is in the usual case of connected graphs.  Our main result
is Theorem~\ref{thm: main}, which states that~$i$-chains of large enough degree
are winnable.  Unlike for graphs, it turns out that all~$i$-chains of a given
degree may be winnable even though there exists an~$i$-chain of larger degree
that is not (cf.~Example~\ref{example: counterintuitive}).  Corollary~\ref{cor:
  minimal winnable} says this will not occur if the Hilbert basis~$\Hilb_i$
  consists of~$0$-$1$ vectors.

For divisors on a graph, the degree function, $\textrm{deg}\colon\Z V\to\Z$, is
a linear function with the following two properties:
\begin{align*}
  \text{invariance under linear equivalence:}&\quad D\sim D' \Rightarrow \deg(D)
  = \deg(D'),\\
  \text{nonnegativity on effective divisors:}&\quad
E\geq 0\Rightarrow\deg(E)\geq0.
\end{align*}
To generalize the notion of degree to higher dimensions, for each~$i$, we look for a linear
function $\deg\colon C_i(\Delta)\to\Z$ with the above two properties.  Any
such linear function can be represented by $\sigma \mapsto \langle \sigma,\sigma'\rangle$
for a fixed $\sigma' \in C_i(\Delta)$, where $\langle \sigma, \sigma' \rangle := \sum_{f \in
\Delta_i} \sigma(f)\sigma'(f)$. To have invariance under linear equivalence,
$\sigma'$ must lie
in the kernel of~$L_i$.   For the function to be nonnegative on effective
chains, $\sigma'$ must itself be effective. Thus, an integer-valued linear function
has our two desired properties if and only if it is expressible as the inner
product with an effective~$i$-chain in~$\ker L_i$.  But no particular one of
these functions stands out as a preferred choice.  Instead, we will take our
generalization to contain the information of the output of all such functions,
as we now describe.

The set~$C:=\left\{ v\in\R^{f_i}:L_iv\geq0\text{ and } v\geq0
\right\}$ is a pointed, rational, polyhedral cone.  Therefore, its set of integer
points, $C\cap \Z^{f_i}$, has a unique Hilbert basis~$\Hilb$
(\cite{Hilbert}, \cite{Schrijver}).  This means that~$C\cap \Z^{f_i}$ is exactly the
set of nonnegative integer linear combinations of~$\Hilb$, and~$\Hilb$ is the smallest
subset of~$C\cap \Z^{f_i}$ with this property. We can now give our definition of degree:

\begin{definition}\label{def: degree} Let~$i\in\Z$.  The {\em $i$-th nonnegative kernel}
  for~$\Delta$ is the monoid
  \[
    \ker^{+}L_i:=\left\{ \sigma\in\ker L_i: \sigma(f)\geq0\ \text{for all
      $f\in \Delta_i$}
    \right\}.
  \]
  Fix an ordering
  \[
    \Hilb_i=\Hilb_i(\Delta)=(h_1,\dots,h_{\ell_i})
  \]
  for the elements of the Hilbert basis for $\ker^{+}L_i$.  The {\em degree}
  of~$\sigma\in C_i(\Delta)$ is 
  \[
    \deg(\sigma):= \deg_i(\sigma) := (\sigma\cdot h_1,\dots, \sigma\cdot
    h_{\ell_i})
  \]
  where~$\sigma\cdot h_j:=\sum_{f\in \Delta_i}\sigma(f)h_j(f)$.
\end{definition}

\begin{remark}
Another possible definition for the degree function is to replace~$\Hilb_i$ in the
definition with a list of only those elements of the Hilbert basis that
are rays of the cone~$L_i^+\otimes\R$.  Denoting this variant of the
definition of degree by~$\rdeg$, we have
\[
  \deg(\sigma)\geq\deg(\sigma')\ \Longleftrightarrow\
  \rdeg(\sigma)\geq\rdeg(\sigma')
\]
for~$\sigma,\sigma'\in C_i(\Delta)$.  This means that all our results relating
winnability of the dollar game to the degree of a chain will hold using either
definition.  One advantage of~$\rdeg$ over~$\deg$ is that it is easier to compute.  
\end{remark}

For each~$i$, our definition of degree is a linear function into~$\Z^{\ell_i}$,
where~$\ell_i$ is the number of elements in~$\Hilb_i$,
and satisfies the two essential properties described earlier: invariance under
linear equivalence is shown below, and nonnegativity on effective chains is
obvious.  It also specializes to the usual definition of degree in the case of a
connected graph, as the Hilbert basis in that case is the sum of all of the
vertices of the graph.

\begin{prop}\label{prop: invariance} The degree of an~$i$-chain depends only on
  its linear equivalence class. 
\end{prop}
\begin{proof} It suffices to show that every element of~$\im L_i$ has degree
  zero.  If~$\tau\in\ker L_i$ and~$\sigma\in C_i(\Delta)$, then 
  \[
    \langle \tau,L_i\sigma\rangle
    =\langle L_i^t\tau,\sigma\rangle
    =\langle L_i\tau,\sigma\rangle
    =0,
  \]
  since~$L_i$ is symmetric. In particular,~$\langle\tau,L_i\sigma\rangle=0$ for
  all~$\tau\in\ker^{+}L_i$.
\end{proof}
\begin{cor}\label{cor: winnable implies nonnegative}
  If an~$i$-chain~$\sigma$ is winnable, then~$\deg(\sigma)\geq0$.
\end{cor}
\begin{proof}  If~$\sigma$ is winnable, then~$\sigma\sim\tau$ for some~$\tau\geq0$.
  Then~$\deg(\sigma)=\deg(\tau)$, and since
  each element of the Hilbert basis~$\Hilb_i(\Delta)$ has nonnegative
  coefficients,~$\deg(\tau)\geq0$.
\end{proof}
\begin{remark}\label{remark: homology invariant}
  Using~\eqref{eqn: ker L_i}, below, the proof of Proposition~\ref{prop:
  invariance} is easily modified to show that every element
  of~$\im\partial_{i+1}$ has degree zero.  Thus, we get the stronger result that
  degree is a homology invariant.
\end{remark}

\begin{definition} A vector~$\delta\in\Z^{|\Hilb_i|}$ is a
  {\em realizable~$i$-degree} if there exists an~$i$-chain~$\sigma$ such
  that~$\deg(\sigma)=\delta$. 
\end{definition}

It is typically the case that not all degrees are realizable. For instance,
consider the~$3$-simplex with single facet~$\overline{1234}$.  In this case, the Hilbert
basis for~$\ker^{+}L_2$, computed by Sage~(\cite{Sage}), is 
\[
  \{\overline{123} + \overline{124}, \overline{123} +\overline{234}, \overline{134} + \overline{124}, \overline{134} + \overline{234}\}.
\]  
Ordering these elements as listed, it is easy to check that there are
no~$2$-chains of degree~$(0,0,0,1)$.

In general, the set of realizable~$i$-degrees forms an additive
monoid~$\mathcal{M}_i(\Delta)$, and Proposition~\ref{prop: invariance} says that
the~$i$-class group~$\class_i(\Delta)$ is graded by~$\mathcal{M}_i$.
Given~$\delta\in\mathcal{M}_i(\Delta)$, let~$\class_i^{\delta}(\Delta)$ denote
the~$\delta$-th graded part of~$\class_i(\Delta)$.  Then there is a faithful
action of the group~$\class_i^0(\Delta)$ on~$\class_i^{\delta}(\Delta)$ given by
addition of~$i$-chains.

\subsection{The group of chain classes of degree zero}
Our next goal is Theorem~\ref{thm: deg 0 iso}, identifying the group of
degree zero~$i$-chains modulo firing rules with the torsion part of the critical
group~$\crit_i(\Delta)$, and thus generalizing a well-known result from the
divisor theory of graphs (cf.~Example~\ref{example: graphs}).  Letting~$K =
\Z$,~$\Q$, or~$\R$, we use the standard notation $X^{\perp}=\left\{ y\in
K:x\cdot y=0 \text{ for all~$x\in X$} \right\}$ for the perpendicular space for
a subset~$X\subseteq K^n$.

By standard linear algebra,
\begin{equation}\label{eqn: ker L_i}
  \ker L_i=\ker\partial_{i+1}\partial_{i+1}^t=\ker\partial_{i+1}^t.
\end{equation}
Using the chain property of boundary maps, we identify a useful
subset of the kernel:
\[
  \im\partial_i^t\subseteq\ker\partial_{i+1}^t=\ker L_i.
\]
If~$f$ is an~$(i-1)$-face of~$\Delta$, the
element~$\partial_i^t(f)$ is called the {\em star} of~$f$; it is a signed sum of
the faces radiating from~$f$. If~$f=\overbar{v_0\cdots v_{i-1}}$,
then each element in the support of its star has the form~$\overbar{v_0\cdots
v_k v v_{k+1}\cdots v_{i-1}}$ for some vertex~$v$.
The set of stars generates~$\im\partial_i^t$.

\begin{lemma}\label{lemma: positive element} For each~$i$, there is a strictly
  positive element~$\tau\in\ker L_i$, i.e., such that~$\tau(f)>0$ for
  all~$f\in\Delta_i$.
\end{lemma}
\begin{proof} For the sake of contradiction, assume no such element~$\tau$
  exists.  Then for every~$\sigma\in\ker L_i$, let~$m_{\sigma}$ denote the least
  (in lexicographic ordering)~$i$-face such that~$\sigma(m)\leq 0$.  Choose
  a~$\sigma\in\ker L_i$ with maximal~$m_{\sigma}$.
  Say~$m:=m_{\sigma}=\overbar{v_0\cdots v_i}$, and consider the star $S :=
  \partial_i^t(\overbar{v_1\cdots v_i})$. The coefficient of~$m$ in~$S$ is~$1$,
  and if~$m_0$ is an~$i$-face such that~$m_0 < m$, then~$m_0$ begins with a
  vertex~$v$ smaller than~$v_1$, meaning one of two cases occurs: either $m_ 0 =
  \overbar{vv_1\cdots v_i}$, in which case the coefficient of~$m_0$ in~$S$ is 1,
  or~$m_0$ does not contain $\overbar{v_1\cdots v_i}$ as a subface, and the
  coefficient of~$m_0$ in~$S$ is~$0$. Either way, if~$m_0<m$, then the
  coefficient of~$m_0$ in~$S$ is nonnegative.  Now consider~$\sigma':= \sigma +
  (1-\sigma(m))S$. Then~$\sigma'\in\ker L_i$, and~$\sigma'(f)>0$ for all
  faces~$f\leq m$, contradicting the maximality of~$m$. So our assumption must
  be false.
\end{proof}

The following is an immediate consequence:
\begin{cor}\label{cor: effective degree 0}
  If~$\sigma$ is an effective~$i$-chain and~$\deg(\sigma)=0$, then~$\sigma=0$. 
\end{cor}

\begin{cor}\label{cor: equal perps} For each~$i$, the~$\Z$-span of~$\ker^{+}L_i$ is~$\ker L_i$.
Hence, 
\[
  (\ker^{+}L_i)^{\perp}=(\ker L_i)^{\perp}=(\ker\partial_{i+1}^t)^{\perp}.
  \]
\end{cor}
\begin{proof}
  Take a strictly positive element~$\tau\in\ker L_i$ that is primitive, i.e., it is not an integer multiple of any other element.  We can then
  complete~$\left\{ \tau \right\}$ to a basis~$\left\{ \tau,
  \sigma_1,\dots,\sigma_k\right\}$ for~$\ker L_i$.
  (To see this, consider the exact sequence
  \[
    0\to \Z\tau\to\Z^n\to\Z^n/\Z\tau\to0.
  \]
Since~$\Z^n/\Z\tau$ is torsion-free, the sequence splits.) Then, for
each nonzero~$N\in\Z$, the set
\[
  \{ \tau, \sigma_1+N\tau,\dots,\sigma_k+N\tau\}
\]
is still a basis for~$\ker L_i$.  By taking~$N\gg 0$, this basis will consist
solely of elements~$\ker_i^+L_i$.
\end{proof}

\begin{theorem}\label{thm: deg 0 iso} For each~$i$, the group of~$i$-chains of
  degree zero modulo firing rules is isomorphic to the torsion part of the~$i$-th
  critical group of~$\Delta$: 
\[
  (\ker L_i)^{\perp}/\im(L_i)= \T(\crit_i(\Delta)).
\]
\end{theorem}
\begin{proof} 
  To see that~$\im L_i\subseteq(\ker L_i)^{\perp}$, let~$\sigma\in\Z\Delta_i$
  and~$\tau\in\ker L_i=\ker\partial_{i+1}^t$.  Then
  \[
    \langle \tau,L_i\sigma\rangle
    =\langle \tau,\partial_{i+1}\partial_{i+1}^t\sigma\rangle
    =\langle \partial_{i+1}^t\tau,\partial_{i+1}^t\sigma\rangle
    =\langle 0,\partial_{i+1}^t\sigma\rangle
    =0.
  \]
  We also have~$(\im \partial_i^t)^{\perp}\subseteq\ker\partial_i$.  To see
  this, take~$\sigma\in(\im\partial_i^t)^{\perp}$ and~$\tau\in\Z\Delta_{i-1}$.
  Then
  \[
    0=\langle \sigma,\partial_i^t\tau\rangle=\langle\partial_i\sigma,\tau\rangle.
  \]
  Since~$\tau$ is arbitrary,~$\partial_i\sigma=0$.

  Next, 
  \[
    \im\partial_{i}^t\subseteq\ker\partial_{i+1}^t
    \Rightarrow
   (\ker L_i)^{\perp}=(\ker\partial_{i+1}^t)^{\perp}\subseteq(\im\partial_i^t)^{\perp}\subseteq\ker\partial_i.
  \]
  Hence,
  \[
    (\ker L_i)^{\perp}/\im L_i\subseteq\ker\partial_i/\im L_i=:\crit_i(\Delta).
  \]
  Since~$\dim_{\Q}(\ker L_i)^{\perp}=\dim_{\Q}(\im L_i)$, the group $(\ker
  L_i)^{\perp}/\im L_i$ is finite, and hence torsion.  So it is a subset
  of~$\T(\crit_i(\Delta))$.  To show the opposite inclusion,
  let~$\sigma\in\ker\partial_i$, and suppose there exists a positive
  integer~$k$ such that~$k\sigma\in\im L_i$.  Say~$k\sigma=L_i\tau$, and
  let~$\nu\in\ker L_i=\ker\partial_{i+1}^t$.  Then
  \[
    k\langle \nu,\sigma\rangle
    =\langle \nu,k\sigma\rangle
    =\langle \nu,L_i\tau\rangle
    =\langle \partial_{i+1}^t\nu,\partial_{i+1}^t\tau\rangle
    =0.
  \]
  Therefore,~$\langle\nu,\sigma\rangle=0$. So each torsion element
  of~$\crit_i(\Delta)$ is an element of~$(\ker L_i)^{\perp}/\im(L_i)$.
\end{proof}

\begin{cor}\label{cor: critical group and homology}
The natural
  surjection~$\crit_i(\Delta)\to\tH_i(\Delta)$ is an isomorphism
  when restricted to the free parts of~$\crit_i(\Delta)$ and~$\tH_i(\Delta)$
  and a surjection when restricted to the torsion parts.
\end{cor}
\begin{proof} Consider the exact sequence
  \[
    0\to\im\partial_{i+1}/\im L_i\to\crit_i(\Delta)\to\tH_i(\Delta)\to0.
  \]
  We have
  \[
    \im L_i\subseteq\im\partial_{i+1}\subseteq(\ker L_i)^{\perp},
  \]
  where the second inclusion follows by an argument similar to that given
  for~$\im L_i$ at the beginning of the proof of Theorem~\ref{thm: deg 0 iso}.
  From Theorem~\ref{thm: deg 0 iso}, it follows that~$\im\partial_{i+1}/\im
  L_i$ is finite.  Tensoring the sequence by~$\Q$ then gives the result about
  the free parts, and since the torsion functor~$\T(\,\cdot\,)$ is
  left-exact, there is a surjection for the torsion parts.
\end{proof}

\begin{remark}\label{remark: computing classes} Let~$\delta$ be a
  realizable~$i$-degree, and fix any~$\sigma\in C_i(\Delta)$ such
  that~$\deg(\sigma)=\delta$.  Then there is a bijection of chain class
  groups~$\class_i^0(\Delta)\to\class_i^{\delta}(\Delta)$ given
  by~$\omega\mapsto\omega+\sigma$ for each~$\omega\in\class_i^0(\Delta)$.  By
  Theorem~\ref{thm: deg 0 iso}, the group~$\class_i^0(\Delta)$ is the torsion
  part of the (finitely-generated abelian group)~$\crit_0(\Delta)$ and hence is
  finite.  Thus, there are only finitely many chains to check to determine
  whether all chains of a given degree are winnable.
\end{remark}

\begin{example}[Graphs]\label{example: graphs} Consider again how our structures
  generalize those on graphs.  In the case~$d=1$, the simplicial
  complex~$\Delta$ is determined by its~$1$-skeleton, a graph~$G$.  We have two
  notions of degree for an element~$\sigma\in C_{i}(\Delta)$: as a~$0$-chain
  on~$\Delta$, there is the degree determined by dot products with elements of
  the Hilbert basis~$\Hilb_0$; and as a divisor on a graph, there is the usual
  degree given by~$\partial_0(\sigma)=\sum_{v\in V}\sigma(v)$. Call the former
  the~$\Delta$-degree,~$\deg(\Delta,\sigma)$, of~$\sigma$, and call the latter
  the~$G$-degree,~$\deg(G,\sigma)$. 

  By definition, the Picard group~$\Pic(G)$ is the set of~$0$-chains modulo the image
  of~$L_0$, and hence, coincides with the~$0$-th class group~$\class_0$.
  Now,~$\Pic(G)$ is graded by~$G$-degree, and its~$G$-degree zero part is by
  definition the Jacobian group~$\Jac(G)$.  Hence,
  \[
   \Jac(G)=\crit_0(\Delta)=\ker\partial_0/\im L_0.
  \]
  On the other hand,~$\class_0(\Delta)$ is graded by~$\Delta$-degree.
  While~$\Pic(G)=\class_0(\Delta)$ as groups, in the case where~$G$ is not
  connected, their gradings differ.

  If~$G$ is connected or, equivalently,~$\tb_0(\Delta)=0$, the Hilbert
  basis~$\Hilb_0$ consists of the all-ones vector~$\vec{1}$,
  and~$\deg(\Delta,\sigma)=\sigma\cdot\vec{1}=\partial_0(\sigma)=\deg(G,\sigma)$.
  Thus,~$\Pic(G)=\class_0$ as graded groups, and~$\Jac(G)$ is the collection
  of~$\Delta$-degree zero~$1$-chains.  As is well-known, the matrix-tree theorem
  implies that~$|\Jac(G)|$ is the number of spanning trees of~$G$.  So~$\Jac(G)$
  is finite, hence torsion, in agreement with Theorem~\ref{thm: deg 0 iso}.

  Now consider the case where~$G$ is not connected.  To fix ideas,
  say~$G$ is the graph consisting of the disjoint union of two triangles, one
  with vertices~$1,2,3$ and the other with vertices~$4,5,6$.  In this case,
  \[
    \Jac(G)=\crit_0(\Delta)\simeq \Z/3\Z\oplus\Z/3\Z\oplus\Z.
  \]
  The Hilbert basis~$\Hilb_0$ consists of two elements~$h_1=(1,1,1,0,0,0)$
  and~$h_2=(0,0,0,1,1,1)$. So if~$\sigma\in C_0(\Delta)$, then
  \[\textstyle
    \deg(G,\sigma)=\sum_{i=1}^6\sigma_i\quad\text{and}\quad
    \deg(\Delta,\sigma)=(\sum_{i=1}^3\sigma_i,\sum_{i=4}^{6}\sigma_i).
  \]
  For instance, if~$\sigma=\overbar{1}-\overbar{4}=(1,0,0,-1,0,0)$,
  then~$\deg(G,\sigma)=0$ while~$\deg(\Delta,\sigma)=(1,-1)\neq(0,0)$. 
  The~$\Delta$-degree zero part of~$\class_0$ is isomorphic to the direct
  sum of two copies of the Jacobian group of a triangle, i.e.,
  to~$\Z/3\Z\oplus\Z/3\Z$.
\end{example}

\subsection{Degree/winnability condition}
We now show that if the degree of an~$i$-chain is sufficiently large, it is
winnable.  The proof requires the following lemma:

\begin{lemma}\label{lemma: main}
  For each integer~$i$, there exists a finite set of~$i$-chains $\mathcal{P}_i$
  such that any~$\sigma\in C_i(\Delta)$ with $\deg(\sigma) \geq 0$ can be written as
  $\sigma = \zeta + \tau + \phi$ where $\deg(\zeta) = 0$, $\tau$ is effective,
  and $\phi\in \mathcal{P}_i$.
\end{lemma}
\begin{proof}
  Having ordered~$\Delta_i$ lexicographically, we make the
  identification~$C_i(\Delta,\R)\simeq\R^{f_i}$ where~$f_i:=|C_i(\Delta)|$.
  Let~$L_i^{\R}:=L_i\otimes\R\colon \R^{f_i}\to\R^{f_i}$, and
  let~$\mathcal{O}^{+}$ be the nonnegative orthant of~$\R^{f_i}$.
  Using dual cones, the fact that~$\sigma$ has degree at least 0 can be
  expressed as follows:
\[
  \sigma\in((\ker L_i^{\R})\cap \mathcal{O}^+)^*\cap\mathbb{Z}^{f_i} 
  =
  ((\ker L_i^{\R})^* + (\mathcal{O}^+)^*)\cap \mathbb{Z}^{f_i}
  =
  ((\ker L_i^{\R})^* + \mathcal{O}^+)\cap \mathbb{Z}^{f_i}.
\]
We can split both $(\ker L_i^{\R})^*$ and $\mathcal{O}^+$ into the Minkowski sum
of the integer points they contain and their respective fundamental
parallelepipeds~$P_1$ and~$P_2$ (with respect to any choice of integral
generators), to get
\begin{align*}
  ( (\ker L_i^{\R})^* + \mathcal{O}^+)\cap \mathbb{Z}^{f_i} 
    &= 
    ( ((\ker L_i^{\R})^*\cap \mathbb{Z}^{f_i} + P_1)+
      (\mathcal{O}^+\cap\mathbb{Z}^{f_i} + P_2))\cap \mathbb{Z}^{f_i}\\
&=   
(\ker L_i^{\R})^*\cap \mathbb{Z}^{f_i}  + \mathcal{O}^+\cap\mathbb{Z}^{f_i} +
(P_1 + P_2)\cap \mathbb{Z}^{f_i}.
\end{align*}
Since~$\ker L_i^{\R}$ is a linear space, $(\ker L_i^{\R})^*=(\ker L_i^{\R})^{\perp}$.
Hence, $(\ker L_i^{\R})^*\cap \mathbb{Z}^{f_i}$ is the set of all~$i$-chains of
degree~$0$, and~$\mathcal{O}^+\cap\mathbb{Z}^{f_i}$ is the set of
effective~$i$-chains.  So letting $\mathcal{P}_i =
(P_1+P_2)\cap\mathbb{Z}^{f_i}$, which is a finite set
since~$P_1$ and~$P_2$ are bounded, completes the proof.
\end{proof}

\begin{theorem}\label{thm: main}  If the degree of a chain is sufficiently large, then it is
  winnable: for each integer~$i$ there exists a
  realizable~$i$-degree~$\delta\in\Z^{|\Hilb_i|}$ such
  that for all~$\sigma\in C_i(\Delta)$, if~$\deg(\sigma)\geq\delta$, then~$\sigma$ is
  winnable.
\end{theorem}
\begin{proof} Let $\mathcal{S}$ be a set of representatives
  for~$\T(\crit_i(\Delta))$, and let~$\mathcal{P}_i$ be as in Lemma~\ref{lemma:
  main}.  By finiteness of~$\mathcal{S}$ and~$\mathcal{P}_i$, there exists
  an~$i$-chain $\omega$ such that the chain $\omega+\gamma+\phi$ is effective
  for all $\gamma\in \mathcal{S}$ and $\phi\in\mathcal{P}_i$.
  Set~$\delta=\deg(\omega)$, and let
  $\sigma$ be an~$i$-chain such that $\deg(\sigma) \geq \delta$. Then
  $\deg(\sigma-\omega) \geq 0$, so by Lemma~\ref{lemma: main} we can write
\[
  \sigma-\omega = \zeta+\tau+\phi
\]
where~$\deg(\zeta)=0$,~$\tau$ is effective, and~$\phi\in\mathcal{P}_i$.  Since
$\deg(\zeta)=0$, we have~$\zeta\in(\ker^{+}L_i)^{\perp}=(\ker L_i)^{\perp}$ by
Corollary~\ref{cor: equal perps}.  So by Theorem~\ref{thm: deg 0 iso}, there
exists~$\gamma\in\mathcal{S}$ such that $\zeta\sim\gamma$.  It follows
that~$\sigma$ is winnable:
\[
  \sigma\sim (\omega+\gamma+\phi)+\tau\geq0.\qedhere
\]
\end{proof}

Let~$\mathcal{W}_i$ be the set of all~$\delta$ satisfying the conditions in
Theorem~\ref{thm: main}.  Then~$\mathcal{W}_i$ is partially ordered
(\S\ref{sect: partial order}) and bounded below by~$0\in\Z^{|\Hilb_i|}$.  So it
is natural to consider its set of minimal elements,~$\min(\mathcal{W}_i)$.  To
see that~$\min(\mathcal{W}_i)$ is finite, consider the polynomial ideal
generated by the monomials~$x^{\delta}:=\prod_ix_i^{\delta_i}$ as~$\delta$
varies over~$\mathcal{W}_i$.  By the Hilbert basis theorem, this ideal is
finitely generated, and its minimal set of generators corresponds
with~$\min(\mathcal{W}_i)$.  See Example~\ref{example: minimal degrees} for the
computation of~$\min(\mathcal{W}_1)$ for a hollow tetrahedron.

Intuition coming from the dollar game on graphs may not apply to~$\mathcal{W}_i$
on a general simplicial complex.  For instance, as in Example~\ref{example:
minimal degrees}, there are typically infinitely many nonnegative realizable
degrees that are not in~$\mathcal{W}_i$.  Further, as will be demonstrated in
Example~\ref{example: counterintuitive}, it may be the case that all~$i$-chains
of a particular realizable degree~$\delta$ are winnable even though there exists
an unwinnable~$i$-chain~$\sigma$ with~$\deg(\sigma)\geq\delta$.  

To finish this section, we describe conditions under
which~$\delta\in\mathcal{W}_i$ if and only if~$\delta$ is realizable and
all~$i$-chains of degree exactly~$\delta$ are winnable.
\begin{prop}\label{prop: effectives exist} Suppose the $i$-th Hilbert
  basis~$\Hilb_i$ of~$\Delta$ consists of $0$-$1$ vectors, and let~$\sigma$ be
  an~$i$-chain such that~$\deg(\sigma)\geq0$.  Then there exists
  an effective~$i$-chain~$\tau$ (not necessarily linearly equivalent
  to~$\sigma$) such that~$\deg(\tau)=\deg(\sigma)$.
\end{prop}
\begin{proof}
  Suppose the result is false, and let~$\sigma$ be a counterexample of minimal
  degree~$\deg(\sigma)\geq0$ (using the component-wise partial order defined in
  Section~\ref{sect: partial order}).  Note that~$\deg(\sigma)\neq0$.  Using notation
  for dual cones from the proof of Lemma~\ref{lemma: main}, we have
  \[
    \sigma\in(\ker L_{i}^{\R}\cap\mathcal{O}^{+})^*=(\ker
    L_{i}^{\R})^*+\mathcal{O}^{+}=(\ker L_{i}^{\R})^{\perp}+\mathcal{O}^{+}.
  \]
  The last equality follows because~$\ker L_{i}^{\R}$ is a linear space.  Therefore,
  over~$\R$, we have~$\sigma=\nu+\tau$ where~$\nu\in(\ker L_{i}^{\R})^{\perp}$
  and~$\tau=\sum_{f\in\Delta_{i}}\tau(f)f$ with~$\tau(f)\geq0$ for
  all~$f\in\Delta_{i}$.  So~$\tau\cdot h=\sigma\cdot h$ for all~$h\in\Hilb_i$, and
  since~$\deg(\sigma)\neq0$, there exists a face~$f'$ such that~$\tau(f')>0$.
  To compute the degree of the integral chain~$\sigma-f'$,
  let~$h=\sum_{f\in\Delta_{i}}h(f)f$ be an arbitrary element of~$\Hilb_{i}$.
  Since~$h(f')\in\left\{ 0,1 \right\}$, taking dot products,
\[
(\sigma-f')\cdot h=
  (\tau-f')\cdot h=\sum_{f\in\Delta_{i}}\tau(f)h(f) - h(f')
=\sum_{f\neq f'}\tau(f)h(f) +(\tau(f')-1)h(f')>-1.
\]
Since~$(\sigma-f')\cdot h\in\Z$ for all~$h\in\Hilb_{i}$, it follows
that~$\deg(\sigma-f')\geq0$.  On the other hand, by Lemma~\ref{lemma: positive
element}, there exists some~$h\in\Hilb_{i}$ such that~$h(f')>0$, and
therefore~$\deg(\sigma-f')$ is strictly smaller than~$\deg(\sigma)$.
By minimality, there exists an effective integral~$i$-chain~$\rho$
with~$\deg(\rho)=\deg(\sigma-f')$.  But then~$\rho+f'$ is an effective divisor of
degree~$\deg(\sigma)$, contradicting the fact that~$\sigma$ is a counterexample.
\end{proof}

\begin{cor}\label{cor: minimal winnable} Suppose~$\Hilb_i$ consists of~$0$-$1$
  vectors and that there exists a realizable~$i$-degree~$\delta$
  such that every~$i$-chain of degree~$\delta$ is winnable.  Then every~$i$-chain with
  degree at least~$\delta$ is winnable.
\end{cor}
\begin{proof}
Let~$\sigma\in C_{i}(\Delta)$ with~$\deg(\sigma)\geq\delta$. 
By Corollary~\ref{prop: effectives exist}, there exists an effective
chain~$\tau\in C_{i}(\Delta)$ of degree~$\deg(\sigma)-\delta$. Since~$\sigma-\tau$ has
degree~$\delta$, by hypothesis it is linearly equivalent to an effective
chain~$\rho$.  Therefore,~$\sigma\sim \tau+\rho\geq0$, and~$\sigma$ is winnable.
\end{proof}

\section{Pseudomanifolds}\label{sect: pseudomanifolds}
In this section we take~$\Delta$ to be a~$d$-dimensional orientable
pseudomanifold.  References for pseudomanifolds include~\cite{Massey}
and~\cite{Spanier}.
To say that~$\Delta$ is a {\em pseudomanifold} means that it is 
\begin{enumerate}
  \item {\em pure}: each facet has dimension~$d$;
  \item {\em non-branching}: each~$(d-1)$-face is a face of at most two
    facets; and
  \item {\em strongly connected}: if~$\sigma$ and~$\sigma'$ are facets, there
    exists a sequence of facets~$\sigma_0,\dots,\sigma_k$ with~$\sigma_0=\sigma$
    and~$\sigma_k=\sigma'$ such that each pair of consecutive facets~$\sigma_i$
    and~$\sigma_{i+1}$ share a~$(d-1)$-face.
\end{enumerate}
The {\em boundary}~$\partial\Delta$ of~$\Delta$ is the collection
of~$(d-1)$-faces of~$\Delta$ that are faces of exactly one facet.
Since~$\Delta$ is a pseudomanifold, it is a standard result that exactly one of
the following must hold in relative homology:
\begin{enumerate}
  \item[(i)] $H_d(\Delta,\partial\Delta)\approx\Z$ and~$H_{d-1}(\Delta,\partial\Delta)$ is torsion-free.
  \item[(ii)] $H_d(\Delta,\partial\Delta)=0$
    and~$H_{d-1}(\Delta,\partial\Delta)$ has torsion
    subgroup~$\T(H_{d-1}(\Delta,\partial\Delta))\approx\Z/2\Z$.
\end{enumerate}
In our case, we are assuming that~$\Delta$ is an {\em orientable
pseudomanifold}, which by definition means that~(i) holds.  It is then
possible to orient the facets of~$\Delta$ so that the sum of their boundaries is
supported on the boundary of~$\Delta$.  Letting~$f^{(1)},\dots,f^{(m)}\in C_d(\Delta)$
be the facets of~$\Delta$, this means that for each~$i$ we can
choose~$\gamma_i\in\left\{ \pm f^{(i)} \right\}$ and
define~$\gamma=\gamma_1+\dots+\gamma_m$ so that ~$\partial_d(\gamma)$ is
supported on~$\partial\Delta$.  (In particular, if~$\Delta$ has no boundary,
then~$\partial_d(\gamma)=0$.) We call the relative cycle~$\gamma$ a {\em
pseudomanifold orientation} for~$\Delta$. Its class~$[\gamma]\in
H_d(\Delta,\partial\Delta)$ is a choice of generator for the top relative
homology group.  Recall that the simplicial complexes studied in this paper all
come with a fixed underlying orientation as a simplicial complex, upon which the
dollar game depends.  The orientations of the facets~$\gamma_i$ need not agree
with those given by that fixed orientation.

The proof of the following is in the appendix.  It was proved in~\cite{DKM1} for the
case~$\tH_{d-1}(\Delta)=0$ and~$\partial\Delta=\emptyset$.  
\begin{prop}\label{prop: pseudomanifold critical group} Suppose~$\Delta$ is
  a~$d$-dimensional orientable pseudomanifold.  If~$\partial\Delta\neq\emptyset$, 
  \[
    \crit_{d-1}(\Delta)=\tH_{d-1}(\Delta)
  \]
  and otherwise, if~$\Delta$ has no boundary,
  \[
    \crit_{d-1}(\Delta)\simeq\left(\Z/m\Z\right)\oplus\tH_{d-1}(\Delta)
  \]
  where~$m=f_d$ is the number of facets of~$\Delta$.
\end{prop}

To define the degree of a~$(d-1)$-chain on a pseudomanifold~$\Delta$, we need to
compute the Hilbert basis for~$\ker^{+}L_{d-1}$.  Our
main goal for this section is a combinatorial description of this basis.  We
start by defining
the {\em $\gamma$-incidence graph}~$\Gamma=\Gamma(\Delta,\gamma)$ as a
directed graph whose vertices are the oriented facets~$\left\{ \gamma_i
\right\}$.  If~$\partial\Delta\neq\emptyset$, let~$\gamma_0:=0\in C_d(\Delta)$,
and include it, too, as a vertex of~$\Gamma$.  The edges of~$\Gamma$ are in
bijection with the codimension-one faces of~$\Delta$.  To describe them,
let~$\sigma$ be any~$(d-1)$-face and write
\[
  \partial_d^t(\sigma)=\gamma_j-\gamma_i
\]
for uniquely determined~$i$ and~$j$.  (If~$\sigma\in\partial\Delta$, then
one of~$i$ or~$j$ will be~$0$.) Let $\sigma^-:=i$ and~$\sigma^+:=j$.
The directed edge corresponding to~$\sigma$ then starts
at~$\gamma_{\sigma^-}$ and ends at~$\gamma_{\sigma^+}$. See
Figures~\ref{fig: hollow tetrahedron} and~\ref{fig: annulus} for examples.

\begin{theorem}[Hilbert basis for an orientable pseudomanifold]\label{thm:
  pseudomanifold Hilbert basis} Let~$\Delta$ be a pseudomanifold with
  pseudomanifold orientation~$\gamma$.  Then the Hilbert basis for the
  nonnegative kernel~$\ker^{+}L_{d-1}$ is the set of incidence vectors for the
  simple directed cycles of~$\Gamma(\Delta,\gamma)$.
\end{theorem}
\begin{proof} Let~$\tau=\sum_{\sigma}a_{\sigma}\sigma\in C_{d-1}(\Delta)\neq0$.  Then
  $\tau\in\ker L_{d-1}=\ker\partial_d^t$ if and only if
  \[
    0=\partial_d^t(\tau)=\sum_{\sigma}a_{\sigma}(\gamma_{\sigma^+}-\gamma_{\sigma^-}).
  \]
  Requiring~$\tau\in\ker^{+}L_{d-1}$ adds the restriction
  that~$a_{\sigma}\geq0$ for all~$\sigma$, which is equivalent to
  saying that~$\tau$ is a directed cycle in~$\Gamma$.  Then~$\tau$ is simple if and
  only if it is not the sum of two other non-trivial directed cycles, which is
  exactly the requirement that~$\tau$ belong to the Hilbert basis.
\end{proof}

\begin{cor} Suppose~$\delta$ is a realizable~$(d-1)$-degree on the orientable
  pseudomanifold~$\Delta$ of dimension~$d$ and that every~$(d-1)$-chain of
  degree~$\delta$ is winnable.  Then every~$(d-1)$-chain
  with degree at least~$\delta$ is winnable.
\end{cor}
\begin{proof}
  The result follows immediately from Theorem~\ref{thm: pseudomanifold Hilbert
  basis} and~Corollary~\ref{cor: minimal winnable}.
\end{proof}

\begin{example}\label{example: hollow tetrahedron}
  Let~$\Delta$ be the hollow tetrahedron with facets~$\overbar{123}$,
  $\overbar{124}$, $\overbar{134}$, and $\overbar{234}$.  A pseudomanifold
  orientation is given by
  \[
    \gamma = \overbar{132}+\overbar{124}+\overbar{143}+\overbar{234}
    = -\overbar{123}+\overbar{124}-\overbar{134}+\overbar{234}.
  \]
  Both~$\Delta$ and its associated ~$\gamma$-incidence graph~$\Gamma(\Delta,\gamma)$ appear in
  Figure~\ref{fig: hollow tetrahedron}.  The edges of~$\Gamma(\Delta,\gamma)$
  are labeled by the corresponding~$1$-faces of~$\Delta$.
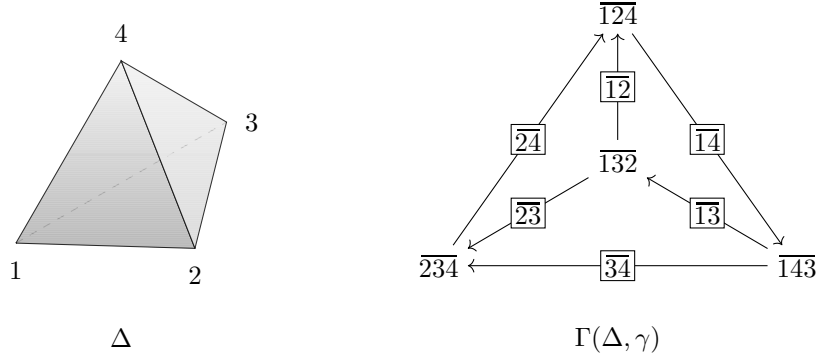
\begin{figure}[ht] 
\begin{tikzpicture} 
  \begin{scope}[yshift=-0.5cm, scale=1.4]
  \node[label={below}:$1$] (1) at (0,0) {};
  \node[label={below}:$2$] (2) at (1.7,-0.05) {};
  \node[label={right}:$3$] (3) at (2.0,1.15) {};
  \node[label={above}:$4$] (4) at (1,1.732) {};
  
  \draw[dashed] (1.center)--(3.center);
  \shadedraw[top color=gray!10, bottom color=gray!60, opacity=0.8] (1.center)--(2.center)--(4.center)--(1.center);
  \shadedraw[top color=gray!10, bottom color=gray!30, opacity=0.8] (2.center)--(3.center)--(4.center)--(2.center);
  \node at (1,-0.9) {$\Delta$};
\end{scope}
  \begin{scope}[xshift=8cm, yshift=0.6cm, scale=1.4]
    \node at (0,0) (1) {$\overbar{132}$};   
    \node at (0,1.4) (3) {$\overbar{124}$};   
    \node at (1.7,-1) (2) {$\overbar{143}$};   
    \node at (-1.7,-1) (4) {$\overbar{234}$};   
    \draw[->] (1) -- (3) node[midway,fill=white, inner sep=0.05cm, draw]
    {$\overbar{12}$}; 
    \draw[->] (3) -- (2) node[midway,fill=white, inner sep=0.05cm, draw]
    {$\overbar{14}$};
    \draw[->] (2) -- (4) node[midway,fill=white, inner sep=0.05cm, draw]
    {$\overbar{34}$};
    \draw[->] (4) -- (3) node[midway,fill=white, inner sep=0.05cm, draw]
    {$\overbar{24}$};
    \draw[->] (2) -- (1) node[midway,fill=white, inner sep=0.05cm, draw]
    {$\overbar{13}$};
    \draw[->] (1) -- (4) node[midway,fill=white, inner sep=0.05cm, draw]
    {$\overbar{23}$};
    \node at (0,-1.7) {$\Gamma(\Delta,\gamma)$};
  \end{scope}
\end{tikzpicture} 
\caption{The hollow tetrahedron and its $\gamma$-incidence graph
(cf.~Example~\ref{example: hollow tetrahedron}).}\label{fig: hollow tetrahedron} 
\end{figure}
The incidence vectors for the three simple directed cycles
of~$\Gamma(\Delta,\gamma)$, and hence the elements of the Hilbert basis
for~$\ker^{+}L_{1}$, are listed as rows in the table below:
\[
  \begin{array}{cccccc}
    \overbar{12}&\overbar{13}&\overbar{14}&\overbar{23}&\overbar{24}&\overbar{34}\\\hline
    1&1&1&0&0&0\\
    0&0&1&0&1&1\\
    0&1&1&1&1&0
  \end{array}.
\]
\end{example}

\begin{example}\label{example: annulus} Figure~\ref{fig: annulus} shows a
  triangulated annulus~$\Delta$ in the plane and its~$\gamma$-incidence graph for the
  counter-clockwise orientation,
  \[
    \gamma =
    \overbar{125}+\overbar{143}+\overbar{154}+\overbar{236}+\overbar{265}+\overbar{346}.
  \]
  The boundary
  is~$\partial\Delta=\left\{\overbar{12},\overbar{13},\overbar{23},\overbar{45},\overbar{46},
  \overbar{56}\right\}$.  Since the boundary is nonempty, the~$\gamma$-incidence
  graph includes the vertex~$*$, representing~$0\in C_{d}(\Delta)$.
  The Hilbert basis for~$\ker^{+}L_1$ has ten elements, two of which are
  displayed below:
\end{example}
\begin{figure}[ht]
\begin{center}
  \begin{tikzpicture}[scale=0.6]
    \def\myir{1.2} 
    \def\myor{3.7} 
    \begin{scope}[yshift=-1.5cm]
    \node[label={below right:$1$}] at (-30:\myor) (1) {};
    \node[label={above:$2$}] at (90:\myor) (2) {};
    \node[label={below left:$3$}] at (210:\myor) (3) {};
    \node[label={[shift={(270:-0.05)}]$4$}] at (270:\myir) (4) {};
    \node[label={[shift={(60:-0.68)}]$5$}] at (30:\myir) (5) {};
    \node[label={[shift={(120:-0.68)}]$6$}] at (150:\myir) (6) {};

    \draw (1.center)--(2.center)--(3.center)--(1.center);
    \draw (4.center)--(5.center)--(6.center)--(4.center);
    \draw[fill=gray!20,draw] (1.center)--(2.center)--(5.center);
    \draw[fill=gray!20,draw] (1.center)--(4.center)--(3.center);
    \draw[fill=gray!20,draw] (1.center)--(5.center)--(4.center);
    \draw[fill=gray!20,draw] (2.center)--(3.center)--(6.center);
    \draw[fill=gray!20,draw] (2.center)--(6.center)--(5.center);
    \draw[fill=gray!20,draw] (3.center)--(4.center)--(6.center);
    \draw (1.center)--(2.center)--(5.center);
    \draw (1.center)--(4.center)--(3.center);
    \draw (1.center)--(5.center)--(4.center);
    \draw (2.center)--(3.center)--(6.center);
    \draw (2.center)--(6.center)--(5.center);
    \draw (3.center)--(4.center)--(6.center);
    \end{scope}
    \begin{scope}[xshift=12cm]
    \node[->] at ({30+0*60}:4) (1) {$\overbar{125}$};
    \node[->] at ({30+1*60}:4) (2) {$\overbar{265}$};
    \node[->] at ({30+2*60}:4) (3) {$\overbar{236}$};
    \node[->] at ({30+3*60}:4) (4) {$\overbar{346}$};
    \node[->] at ({30+4*60}:4) (5) {$\overbar{143}$};
    \node[->] at ({30+5*60}:4) (6) {$\overbar{154}$};
      \node at (0,0) (0) {$*$};

      \draw[->] (2) -- (1) node[midway,fill=white, inner sep=0.05cm, draw]
      {$\overbar{25}$}; 
      \draw[->] (3) -- (2) node[midway,fill=white, inner sep=0.05cm, draw]
      {$\overbar{26}$}; 
      \draw[->] (4) -- (3) node[midway,fill=white, inner sep=0.05cm, draw] {$\overbar{36}$}; 
      \draw[->] (5) -- (4) node[midway,fill=white, inner sep=0.05cm, draw] {$\overbar{34}$}; 
      \draw[->] (6) -- (5) node[midway,fill=white, inner sep=0.05cm, draw]
      {$\overbar{14}$}; 
      \draw[->] (1) -- (6) node[midway,fill=white, inner sep=0.05cm, draw]
      {$\overbar{15}$}; 

      \draw[->] (0) -- (1) node[midway,fill=white, inner sep=0.05cm, draw]
      {$\overbar{12}$}; 
      \draw[->] (2)--(0) node[midway,fill=white, inner sep=0.05cm, draw]
      {$\overbar{56}$};
      \draw[->] (0)--(3) node[midway,fill=white, inner sep=0.05cm, draw]
      {$\overbar{23}$};
      \draw[->] (0)--(4) node[midway,fill=white, inner sep=0.05cm, draw]
      {$\overbar{46}$};
      \draw[->] (5)--(0) node[midway,fill=white, inner sep=0.05cm, draw]
      {$\overbar{13}$};
      \draw[->] (6)--(0) node[midway,fill=white, inner sep=0.05cm, draw]
      {$\overbar{45}$};

    \end{scope}
    \node at (0,-5.5) {$\Delta$};
    \node at (12cm,-5.5) {$\Gamma(\Delta,\gamma)$};
  \end{tikzpicture}
\end{center}
\caption{A triangulated annulus and its $\gamma$-incidence graph
  (cf.~Example~\ref{example: annulus}).}\label{fig: annulus} 
\end{figure}
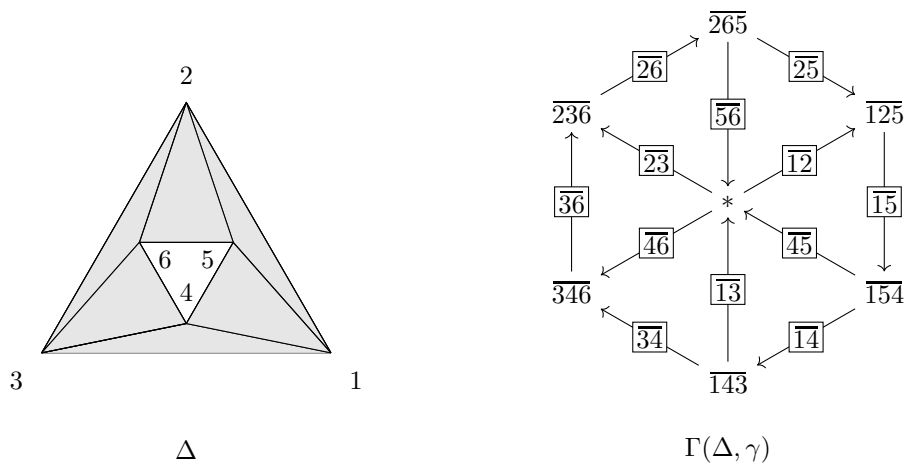
\begin{center}
\begin{tikzpicture}
  \node at (0,0) {$
  \begin{array}{cccccccccccc}
    \overbar{12}&\overbar{13}&\overbar{14}&\overbar{15}&\overbar{23}&\overbar{25}&\overbar{26}
    &\overbar{34}&\overbar{36}&\overbar{45}&\overbar{46}&\overbar{56}\\\hline
    0& 0& 1& 1& 0& 1& 1& 1& 1& 0& 0& 0\\
    0& 0& 0& 1& 1& 1& 1& 0& 0& 1& 0& 0
  \end{array}
$};
\node at (0,-1.3) {Two elements in the Hilbert basis for~$\ker^{+}L_1$.};
\end{tikzpicture}
\end{center}

\begin{example}\label{example: Klein} The condition of being orientable as a
  pseudomanifold is necessary in both Proposition~\ref{prop: pseudomanifold
  critical group} and Theorem~\ref{thm: pseudomanifold Hilbert basis}. The Klein
  bottle simplicial complex in Figure~\ref{fig: Klein} is a non-orientable
  pseudomanifold of dimension~$2$.  Computing with Sage~(\cite{Sage}), we
  find~$\crit_1(\Delta)\simeq\Z/2\Z\oplus\Z/2\Z\oplus\Z$ and that the Hilbert
  basis for~$\ker^{+}L_1$ has~14 elements.  Three of these basis elements are
  not $0$-$1$ vectors and, thus, are not incidence vectors of simple cycles in a
  directed graph.

\begin{figure}[ht] 
\centering
\begin{tikzpicture}[scale=0.9]

\draw (0,0) node[label=270:{$1$},ball] (bl1) {};
\draw (1,0) node[label=270:{$2$},ball] (b2) {};
\draw (2,0) node[label=270:{$3$},ball] (b3) {};
\draw (3,0) node[label=270:{$1$},ball] (br1) {};

\draw (0,3) node[label=90:{$1$},ball] (tl1) {};
\draw (1,3) node[label=90:{$2$},ball] (t2) {};
\draw (2,3) node[label=90:{$3$},ball] (t3) {};
\draw (3,3) node[label=90:{$1$},ball] (tr1) {};

\draw (0,2) node[label=180:{$4$},ball] (l4) {};
\draw (0,1) node[label=180:{$5$},ball] (l5) {};

\draw (3,1) node[label=0:{$4$},ball] (r4) {};
\draw (3,2) node[label=0:{$5$},ball] (r5) {};

\draw (1.5,1) node[label={[label distance=1.0mm]90:{$6$}},ball] (6) {};
\draw (1,2) node[label={[label distance=-1.4mm]135:{$7$}},ball] (7) {};
\draw (2,2) node[label={[label distance=-1.4mm]45:{$8$}},ball] (8) {};

\draw (bl1)--(b2)--(b3)--(br1);
\draw (tl1)--(t2)--(t3)--(tr1);
\draw (bl1)--(l5)--(l4)--(tl1);
\draw (br1)--(r4)--(r5)--(tr1);
\draw (b2)--(l5);
\draw (b2)--(6);
\draw (b3)--(r4);
\draw (b3)--(6);
\draw (l5)--(6);
\draw (l5)--(7);
\draw (r4)--(6);
\draw (r4)--(8);
\draw (r5)--(8);
\draw (l4)--(7);
\draw (l4)--(t2);
\draw (6)--(7);
\draw (6)--(8);
\draw (7)--(8);
\draw (7)--(t2);
\draw (t3)--(7);
\draw (t3)--(8);
\draw (t3)--(r5);
\end{tikzpicture}
\caption{Triangulation of a Klein bottle (cf.~Example~\ref{example:
Klein}).}\label{fig: Klein} 
\end{figure}
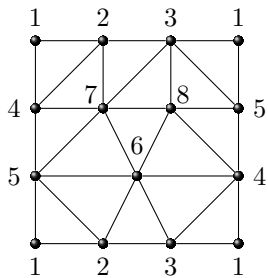
\end{example}

\begin{example}[Computing minimal winning degrees]\label{example: minimal degrees}
Let $\Delta$ be the hollow tetrahedron in Example~\ref{example: hollow
tetrahedron}, and use lexicographic ordering of the edges of~$\Delta$
to identify~$C_1(\Delta)$ with~$\Z^6$, as usual.  For the purpose of computing
degrees, we can order the elements of the Hilbert basis~$\Hilb_1$ for $\Delta$,
computed in Example~\ref{example: hollow tetrahedron}, as
\[
  h_1=(1,1,1,0,0,0),\quad h_2=(0,0,1,0,1,1),\quad h_3=(0,1,1,1,1,0).
\]
By Theorem~\ref{thm: main}, there exists an effective~$1$-chain~$\tau\in\Z^6$
such that every~$1$-chain of degree at least~$\delta:=\deg(\tau)$ is winnable.
In this example, we compute all minimal such~$\delta$ (the set
$\min(\mathcal{W}_i)$, using earlier notation).  We then exhibit an infinite
family of nonnegative realizable~$1$-degrees that are not realizable by
winnable~$1$-chains.

Choose an effective~$\tau\in C_1(\Delta)=\Z^6$ with~$\deg(\tau)= \delta$, and
suppose that every~$1$-chain of degree at least~$\delta$ is winnable.  Let
$\sigma^{(0)}, \sigma^{(1)}, \sigma^{(2)}, \sigma^{(3)}$ be representatives for
the elements of~$\crit_1(\Delta)\simeq\Z/4\Z$.  Then the equivalence classes
of~$1$-chains of degree~$\delta$ in~$\class_1:=C_1(\Delta)/\im L_1$ are
$\tau+\sigma^{(i)}$ for~$i=0,\dots,3$ (cf.~Remark~\ref{remark: computing
classes}).  Each~$\sigma^{(i)}$ has degree~$0$ by Theorem~\ref{thm: deg 0 iso}.
By assumption $\tau+\sigma^{(i)}$ is winnable, so working modulo~$\im L_1$, we
can choose the~$\sigma^{(i)}$ so that each~$\tau+\sigma^{(i)}$ is effective.  In
order to minimize~$\delta$, we minimize~$\tau$.

First, suppose $\delta_1 = 0$. Since $\tau$ is effective and $\tau\cdot
h_1=\tau_1+\tau_2+\tau_3 = \delta_1=0$, it follows that $\tau_1=\tau_2=\tau_3 =
0$.  Using this, it similarly follows that $\sigma^{(i)}_1 = \sigma^{(i)}_2 =
\sigma^{(i)}_3 = 0$ for $i=0,1,2,3$.  Some linear algebra shows
that~$\crit_1(\Delta)$ is generated by~$(0,0,0,1,-1,1)$ and~$1$-chains in the
image of the Laplacian which are 0 in the first three components are exactly
those of the form $(0,0,0,4k, -4k, 4k)$ for some integer~$k$.  So up to
re-indexing, $\sigma^{(i)} = (0,0,0, i+4k_i, -i-4k_i, i+4k_i)$ for some
integers~$k_i$.  Now consider the conditions on~$\tau$, besides~$\tau\geq0$,
required to ensure each~$\tau+\sigma^{(i)}$ is effective.  These are 
\[
\left(
\begin{array}{c}
\tau_4\\\tau_5\\\tau_6
\end{array}
\right)
\geq
\left(
\begin{array}{r}
-i\\i\\-i
\end{array}
\right)
+
k_i
\left(
\begin{array}{r}
-4\\4\\-4
\end{array}
\right)
\]
for some integer~$k_i$ and for~$i=0,\dots,3$.  For~$i=0$, we take~$k_i=0$ and
see there is no additional condition imposed on~$\tau$; for~$i=1$,
either~$\tau_5\geq1$ or both~$\tau_4$ and~$\tau_6$ are at least~$3$; for~$i=2$,
either~$\tau_5\geq2$ or both~$\tau_4$ and~$\tau_6$ are at least~$2$; and
for~$i=3$, either~$\tau_5\geq3$ or both~$\tau_4$ and~$\tau_6$ are at least~$1$.
Thus, to minimize~$\tau$, there are eight cases to consider.  In all of
these,~$\deg(\tau)\geq(0,3,3)$.


Next, suppose~$\delta_2=0$.  By a similar argument (or by symmetry, swapping
vertex~$1$ with~$4$ and vertex~$2$ with~$3$), we find minimal~$\tau$ have degree
at least~$(3,0,3)$.  Finally, suppose~$\delta_3=0$.  In that
case,~$\tau_2=\tau_3=\tau_4=\tau_5=0$
and~$\sigma_2^{(i)}=\sigma_3^{(i)}=\sigma_4^{(i)}=\sigma_5^{(i)}=0$ for all~$i$.
However, requiring a chain of the form~$(a,0,0,0,0,b)$ to represent an element
in~$\crit_1(\Delta)$---and hence be in the kernel
of~$\partial_1$---forces~$a=b=0$. That is not possible since the~$\sigma^{(i)}$
are a full set of representatives for~$\crit_1(\Delta)$. So we must
have~$\delta_3\geq 1$.

Combining the above, we conclude~$\delta$ is greater than or equal to one of
~$(0,3,3)$,~$(3,0,3)$, or~$(1,1,1)$.  In fact, these three degrees are minimal
winning degrees for $\Delta$ since there exist four effective~$1$-chains of
each degree that are pairwise not linearly equivalent.  We list these chains in
the table below:
\[
  \begin{array}{c|c}
    \text{degree~$\delta$}&\text{representatives for~$\class_1(\Delta)$}\\\hline
    (0,3,3)&(0,0,0,3,0,3),\ (0,0,0,2,1,2),\ (0,0,0,1,2,1),\ (0,0,0,0,3,0)\\
    (3,0,3)&(3,0,0,3,0,0),\ (2,1,0,2,0,0),\ (1,2,0,1,0,0),\ (0,3,0,0,0,0)\\
    (1,1,1)&(1,0,0,1,0,1),\ (1,0,0,0,1,0),\ (0,1,0,0,0,1),\ (0,0,1,0,0,0)
  \end{array}.
\]

On a graph, there are only finitely many nonnegative degrees realizable by
unwinnable divisors. That is not usually the case for a general simplicial
complex.  For instance, on our current~$\Delta$, consider the family
of~$1$-chains $\sigma=(a,-b,b,0,0,0)$ where~$a\geq0$ and~$b>0$.  We
have~$\deg(\sigma)=(a,b,0)\geq0=(0,0,0)$.  Let~$\tau$ be any effective~$1$-chain
of degree~$(a,b,0)$.  Taking the dot product of~$\tau$ with each~$h_i$, it
follows that ~$\tau=(a,0,0,0,0,b)$, and thus~$\sigma-\tau=(0,-b,b,0,0,-b)$.
However, computing the Hermite normal form for~$L_1$, we see that~$\im L_1$ is
spanned by~$(1,0,-1,3,-2,3)$, $(0,1,-1,1,-1,2)$, and $(0,0,0,4,-4,4)$.  It is
straightforward to check that~$\sigma-\tau\not\in\im L_1$, and
hence~$\sigma\not\sim\tau$.  Hence,~$\sigma$ is not winnable.
\end{example}

\section{Forests}\label{sect: forests}
It is well-known that the dollar game on a graph is winnable for all initial
configurations of degree zero if and only if the graph is a tree (e.g.,
cf.~\cite{Baker}).  In this section, that result is extended to higher
dimensions.  We first recall the basics of trees on simplicial complexes as
developed by Duval, Klivans, and Martin in \cite{DKM1} and \cite{DKM2}. In
\cite{DKM1}, it is shown that under certain circumstances, each critical group is
isomorphic to the cokernel of a certain submatrix of the corresponding Laplacian
matrix called the {\em reduced Laplacian}.  Theorem~\ref{thm: reduced Laplacian
iso} generalizes that result by loosening the hypotheses.

\begin{definition}\label{def: spanning forest} 
  A {\em spanning $i$-forest of $\Delta$} is an $i$-dimensional subcomplex
  $\Upsilon\subseteq\Delta$ with $\Skel_{i-1}(\Upsilon)=\Skel_{i-1}(\Delta)$ and
  satisfying the three conditions
  \begin{enumerate}
    \item\label{sst1}  $\tH_i(\Upsilon)=0$;
    \item\label{sst2}  $\tb_{i-1}(\Upsilon)=\tb_{i-1}(\Delta)$;
    \item\label{sst3}  $f_i(\Upsilon) = f_i(\Delta)-\tb_i(\Skel_i(\Delta))$.
  \end{enumerate}
  In the case where~$\tb_{i-1}(\Delta)=0$, a spanning ~$i$-forest is
  called a {\em spanning $i$-tree}.    The complex~$\Delta$ is a {\em forest} if
  it is a spanning forest of itself, i.e., if~$\tH_d(\Delta)=0$.  If, in
  addition,~$\tb_{d-1}(\Delta)=0$, then~$\Delta$ is a {\em tree}.
\end{definition}

\noindent{\bf Remarks.}  Let~$\Upsilon$ be an~$i$-dimensional subcomplex of~$\Delta$ sharing the
same~$(i-1)$-skeleton.
\begin{enumerate}
  \item For a graph~$G$, the above definition says that a (one-dimensional) spanning
    forest contains all of the vertices of~$G$ and: (i) has no cycles, (ii)
    has the same number of components as~$G$, and (iii) has~$m-c$ edges,
    where~$m$ is the number of edges and~$c$ is the number of components
    of~$G$.
  \item The condition~$\tH_i(\Upsilon)=0$ is equivalent to the elements of the
    set
    \[
A:=\left\{\pup{i}(f):f\in\Upsilon_i \right\}
    \]
    being linearly independent (over~$\Z$ or, equivalently, over~$\Q$).
  \item Since~$\Upsilon$ and~$\Delta$ have the same~$(i-1)$-skeleton,
    $\partial_{\Delta,i-1}=\pup{i-1}$, and
    hence,~$\tb_{i-1}(\Upsilon)=\tb_{i-1}(\Delta)$ is equivalent
    to~$\rk\im\pup{i}=\rk\im\partial_{\Delta,i}$. 
  \item It follows from the previous two remarks that~$\Upsilon$ is a
    spanning~$i$-forest if and only if~$A$, defined above, is a basis
    for~$\im\partial_{\Delta,i}$ over~$\Q$, i.e, the columns of the
    matrix~$\partial_{\Delta,i}$ corresponding to the~$i$-faces of~$\Upsilon$
    are a~$\Q$-basis for the column space of~$\partial_{\Delta,i}$.  In
    particular, spanning~$i$-forests always exist.
  \item\label{skelitem} Since~$\partial_{\Delta,j}=\partial_{\Skel_i(\Delta),j}$ for
    all~$j\leq i$, it follows the~$j$-th reduced homology groups, Betti numbers,
    and critical groups for~$\Delta$ and for~$\Skel_i(\Delta)$ are the same for
    all~$j<i$.  In particular, this implies that the~$j$-forests
    (resp.,~$j$-trees) of~$\Delta$ are the same as those for~$\Skel_i(\Delta)$
    for all~$j\leq i$. 
\end{enumerate}

\begin{prop}[{\cite[Prop~3.5]{DKM1}, \cite{DKM2}}] Any two of the three
  conditions defining a spanning~$i$-forest implies the remaining condition.
\end{prop}

The proof of the following is in the appendix.  It generalizes a result
in~\cite{DKM1}, where it is proved with the assumptions that~$\Delta$ is pure,
that~$\tb_i(\Delta)=0$ for all~$i<d$, and
that~$\tH_{i-1}(\Upsilon)=0$.
\begin{theorem}\label{thm: reduced Laplacian iso}  Suppose that~$\Upsilon$ is
  an~$i$-dimensional spanning forest of~$\Delta$ such that
  $\tH_{i-1}(\Upsilon)=\tH_{i-1}(\Delta)$.
  Let~$\Theta:=\Delta_i\setminus\Upsilon_i$.  Define the~{\em reduced
  Laplacian}~$\tilde{L}$ of~$\Delta$ with respect to~$\Upsilon$ to be the square
  submatrix of~$L_{i}$ consisting of the rows and columns indexed by~$\Theta$.
  Then there is an isomorphism
  \[
    \crit_{i}(\Delta)\xrightarrow{\sim}\Z\Theta/\im\tilde{L}
  \]
  obtained by setting the faces of~$\Upsilon_{i}$ equal to~$0$.
\end{theorem} 

\begin{definition}\label{def: forest number}  Define the {\em $i$-complexity} or {\em $i$-forest number} of~$\Delta$ to be
  \[
   \tau:=\tau_i(\Delta):=\sum_{\Upsilon\subseteq\Delta}|\T(\tH_{i-1}(\Upsilon))|^2
  \]
  where the sum is over all spanning~$i$-forests~$\Upsilon$ of~$\Delta$.
\end{definition}
\begin{prop}\label{prop: forest number 1}~$\tau_i(\Delta)=1$ if and only if~$\Skel_i(\Delta)$ is a
  spanning~$i$-forest of~$\Delta$ and~$\tH_{i-1}(\Delta)$ is torsion-free.
  If~$\Skel_i(\Delta)$ is a spanning~$i$-forest, regardless of
  whether~$\tH_{i-1}(\Delta)$ is torsion-free, then~$\Skel_i(\Delta)$ is the unique
  spanning~$i$-forest of~$\Delta$.
\end{prop} 
\begin{proof} Suppose that~$\tau_i(\Delta)=1$.  Then~$\Delta$ possesses a unique
  spanning~$i$-forest~$\Upsilon$, and~$\tH_{i-1}(\Upsilon)$ is torsion-free.
  Considering~$\partial_i$ as a matrix, it follows that its set of columns has a
  unique maximal linearly independent subset: those columns corresponding to the
  faces of~$\Upsilon$.  Since the columns of~$\partial_i$ are all nonzero, it
  must be that the columns corresponding to~$\Upsilon$ are the only columns,
  i.e.,~$f_i(\Upsilon)=f_i(\Delta)$, and hence~$\Upsilon=\Skel_i(\Delta)$.  It
  follows that $\tH_{i-1}(\Delta)=\tH_{i-1}(\Upsilon)$ and hence is
  torsion-free.
  
  Now suppose~$\Skel_i(\Delta)$ is a spanning~$i$-forest and
  let~$\Upsilon\subseteq\Delta$ be any spanning~$i$-forest.  Since
  $\tH_i(\Skel_i(\Delta))=0$, it follows from condition~\ref{sst3} of
  Definition~\ref{def: spanning forest} that
  \[
    f_i(\Upsilon)=f_i(\Delta)-\tb_i(\Skel_i(\Delta))=f_i(\Delta).
  \]
  Hence,~$\Upsilon=\Skel_i(\Delta)$.  So~$\Skel_i(\Delta)$ is the unique
  spanning~$i$-forest of~$\Delta$.  Further, if~$\tH_{i-1}(\Skel_i(\Delta))$ is
  torsion free, then~$\tau_i(\Delta)=|\T(\tH_{i-1}(\Delta))|^2=1$.
\end{proof}

\begin{theorem}[{\cite[Theorem 8.1]{DKM2}}]\label{thm: cut flow sequences}
  $|\T(\crit_{i-1}(\Delta))|=\tau_i(\Delta)$.\footnote{In~\cite{DKM2},
  this theorem is stated only for $i=\dim(\Delta)$.  The version
  stated here follows by restricting to~$\Skel_i(\Delta)$
  (cf.~Remark~\ref{skelitem}).}
\end{theorem}

\begin{cor}\label{cor: t=1} All~$(i-1)$-chains of degree~$0$
on~$\Delta$ are winnable if and only if~$\tau_i(\Delta)=1$.
\end{cor}
\begin{proof}  By Proposition~\ref{prop: invariance} and Corollary~\ref{cor:
  effective degree 0}, an~$(i-1)$-chain of degree~$0$ is winnable if and only if
  it is linearly equivalent to the zero chain. The~$(i-1)$-chains of degree~$0$
  are the elements~$(\ker^{+}L_{i-1})^{\perp}=(\ker L_{i-1})^{\perp}$. Hence, by
  Theorem~\ref{thm: deg 0 iso}, all~$(i-1)$-chains of degree~$0$ are winnable if
  and only if~$\T(\crit_{i-1}(\Delta))=0$.  The result then follows from
  Theorem~\ref{thm: cut flow sequences}.
\end{proof}

\begin{remark}\label{remark: disconnected} As discussed in the introduction,
  Corollary~\ref{cor: t=1} generalizes the result that all divisors of
  degree~$0$ on a graph are winnable if and only if the graph is a tree.
  However,  for graphs, Corollary~\ref{cor: t=1} says that all divisors of
  degree~$0$ on a {\em forest} are winnable.  This apparent contradiction is
  resolved by the fact that for unconnected graphs, our simplicial notion of
  degree differs from the usual one for graphs.  See Example~\ref{example: graphs}.  
\end{remark}

\begin{example}\label{example: real projective plane} Simply being a spanning
  tree is not enough to guarantee winnability of all degree~$0$ divisors.
  Figure~\ref{fig: real projective plane} illustrates a two-dimensional
  complex~$P$ which is a triangulation of the real projective plane.  We
  have~$\tH_0(P)=\tH_2(P)=0$, and~$\tH_1(P)\approx\Z/2\Z$.  Therefore,~$P$ is a
  spanning tree with tree number~$\tau_2(P)=4$.  The
  cycle~$\sigma:=\overline{12}+\overline{23}-\overline{13}$ is a~$1$-chain in
  the image of~$\partial_2$ and hence, by Remark~\ref{remark: homology
  invariant}, has degree $0$.  As argued in the first line of the
  proof of Corollary~\ref{cor: t=1}, if~$\sigma$ were winnable, it would be
  linearly equivalent to the zero chain.  We used Sage~(\cite{Sage}) to find
  that $\crit_1(P)\approx\Z/2\Z\times\Z/2\Z$ and~$\sigma\notin\im L_1$.
  Hence,~$2\sigma$ is winnable, but~$\sigma$ is not.
\begin{figure}[ht] 
  \centering
  \begin{tikzpicture}[scale=0.75]
  \def\r{2.0cm}
  \draw[fill=gray!40] ({\r*cos(0*60)},{\r*sin(0*60)})--({\r*cos(1*60)},{\r*sin(1*60)})--({\r*cos(2*60)},{\r*sin(2*60)})--({\r*cos(3*60)},{\r*sin(3*60)})--({\r*cos(4*60)},{\r*sin(4*60)})--({\r*cos(5*60)},{\r*sin(5*60)})--({\r*cos(6*60)},{\r*sin(6*60)});
  
  \draw ({\r*cos(0*60},{\r*sin(0*60)}) node[label=0*60:{$5$},ball] (1) {};
  \draw ({\r*cos(1*60},{\r*sin(1*60)}) node[label=1*60:{$6$},ball] (2) {};
  \draw ({\r*cos(2*60},{\r*sin(2*60)}) node[label=2*60:{$4$},ball] (3) {};
  \draw ({\r*cos(3*60},{\r*sin(3*60)}) node[label=3*60:{$5$},ball] (11) {};
  \draw ({\r*cos(4*60},{\r*sin(4*60)}) node[label=4*60:{$6$},ball] (22) {};
  \draw ({\r*cos(5*60},{\r*sin(5*60)}) node[label=5*60:{$4$},ball] (33) {};

  \def\s{0.8cm}
  \draw ({\s*cos(210},{\s*sin(210)}) node[label=210:{$1$},ball] (4) {};
  \draw ({\s*cos(-30},{\s*sin(-30)}) node[label=-30:{$2$},ball] (5) {};
  \draw ({\s*cos(90},{\s*sin(90)}) node[label=90:{$3$},ball] (6) {};

  \draw (22)--(4);
  \draw (22)--(5)--(33);
  \draw (5)--(1)--(6)--(2);
  \draw (6)--(3)--(4)--(11);
  \draw (6)--(4)--(5)--(6);
\end{tikzpicture}
\caption{A triangulation of the real projective plane.}\label{fig:
real projective plane} 
\end{figure}
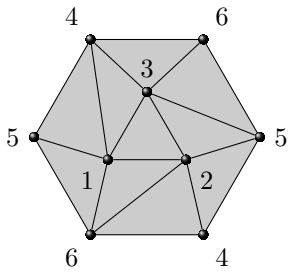
\end{example}

\begin{example}\label{example: counterintuitive} This example demonstrates that
  all~$i$-chains of degree~$0$ of a complex can be winnable, even though there
  are unwinnable~$i$-chains of nonnegative degree.  Let $\Delta$ be the
  three-dimensional simplicial complex with facets
\begin{align*}
&(1, 2, 3, 4), (1, 2, 3, 6), (1, 2, 3, 7), (1, 2, 4, 6), (1, 2, 5, 7),  (1, 3,
4, 7), (1, 3, 5, 7), (1, 4, 5, 6), (1, 4, 5, 7), 
\\&(1, 4, 6, 7), (2, 3, 4, 7), (2, 3, 5, 6), (2, 3, 5, 7), (2, 4, 5, 6), 
(3, 4, 5, 7), (3, 5, 6, 7), (4, 5, 6, 7).
\end{align*}
We have $\tilde{H}_3(\Delta) \cong 0$ and $\tilde{H}_2(\Delta) \cong
\mathbb{Z}$; so by Proposition~\ref{prop: forest number 1}, it follows that
$\Delta$ is a forest with $\tau_3(\Delta) = 1$. Corollary~\ref{cor: t=1} then
implies that all~$2$-chains on $\Delta$ of degree~$0$ are winnable.  

The Hilbert basis of~$\ker^{+}L_2$ for $\Delta$ has~$445$ elements.\footnote{We
used the PyNormaliz package in Sage~(\cite{Sage}) for the Hilbert basis
computations in this example.} Let~$A$ be the matrix whose rows are these Hilbert
basis elements.  Each~$2$-face of~$\Delta$ may be considered as a chain and,
thus, has a degree.  These degrees form the~$33$ columns of~$A$. It follows that
the degrees of all effective~$2$-chains are precisely the nonnegative integer
linear combinations of the columns of~$A$.  The Hilbert basis for the polyhedral
cone generated by the columns of~$A$ consists of the columns of~$A$ and one
other element~$\delta$.  By the characterization of the Hilbert basis,~$\delta$
cannot be realized by any effective two-chain, but using linear algebra it is
possible to find non-effective two-chains of degree~$\delta$, one of which is
\[
(1, 2, 3) - (1, 2, 7) + (1, 3, 5) + (1, 3, 6) + (1, 4, 6) + (1, 6, 7) + (2, 4,
5). 
\]
Thus, the above~$2$-chain is unwinnable but has nonnegative degree.
\end{example}

\subsection{Spanning trees acyclic in codimension one}\label{sect: staco}
\begin{definition}  For each integer~$i$, let
  \[
    \Lambda_i(\Delta) = \Span_{\Z_{\geq0}}\left\{
      \partial_{i+1}(f):f\in\Delta_{i+1}
    \right\}\subset C_{i}(\Delta):=\Z\Delta_{i}.
  \]
  and
  \[
    X_{i}(\Delta):=\left\{ \sigma\in C_{i}(\Delta):\partial_{i}(\sigma)\in\Lambda_{i-1}(\Delta) \right\}.
  \]
\end{definition}
The above definition was introduced by S.~Corry and L.~Keenan
(\cite{CorryKeenan}).  Since~$\Lambda_{-1}(\Delta)=\Z_{\geq0}$
and, therefore,~$X_0(\Delta)=\left\{ \sigma\in C_0(\Delta):\partial_0(\sigma)\geq0 \right\}$,
they regarded the sets~$X_i(\Delta)$ as generalizing the notion of divisors of
nonnegative degree on a graph and explored their relation to the winnability of the
dollar game.  They conjectured the equivalence of~\ref{p1} and~\ref{p2} in the
following proposition and proved it in the case~$i=2$ on a simplicial surface.

\begin{prop}\label{prop: staco} 
  The following are equivalent for~$i\leq d$:
  \begin{enumerate}[leftmargin=*,label=(\arabic*),font=\upshape]
    \item\label{p1} Every~$\sigma\in X_{i-1}(\Delta)$ is winnable.
    \item\label{p2} $\crit_{i-1}(\Delta)=0$.
    \item\label{p3} $\Skel_{i}(\Delta)$ is a spanning~$i$-tree of~$\Delta$ and
  $\tH_{i-1}(\Delta)=0$. 
  \end{enumerate}
  In particular, when~$i=d$, the three conditions are equivalent to~$\Delta$
  being a tree, acyclic in codimension one.
\end{prop}
\begin{proof} We first note that since~$\Delta$ has the standard orientation,
  the only nonnegative element of~$\ker\partial_{i-1}$ is~$0$.  To see this,
  suppose~$\sigma=\sum_{f\in\Delta_i}a_ff\neq0$ with~$a_f\geq0$ for all~$f$.
  Let~$\overbar{v_0\cdots v_i}$ be the lexicographically largest element in the
  support of~$\sigma$ (with~$v_0<\cdots<v_i$). For each~$v\in V$ such
  that~$v\leq v_0$, let~$g_v:=\overbar{vv_1\cdots v_i}$.  Then the coefficient
  of~$\overbar{v_1\cdots v_i}$ in~$\partial_{i-1}(\sigma)$ is~$\sum_{v\in
  V}a_{g_v}>0$.  Hence,~$\sigma\notin\ker\partial_{i-1}$.  We will need this
  fact later in the proof.
  
  Letting~$\mathcal{E}$ denote the set of effective~$(i-1)$-chains, we can
  write~$X_{i-1}(\Delta)=\mathcal{E}+ \ker\partial_{i-1}$.  Thus,~\ref{p1} is
  equivalent to $\mathcal{E}+ \ker\partial_{i-1}\subseteq\mathcal{E}+ \im
  L_{i-1}$, which in turn is equivalent to
  \smallskip

  \begin{enumerate}
    \item[\ref{p1}$^\prime$] $\qquad\displaystyle\mathcal{E}+
      \ker\partial_{i-1}=\mathcal{E}+ \im L_{i-1}$
  \end{enumerate}
  \smallskip

  \noindent since~$\im L_{i-1}\subseteq\ker\partial_{i-1}$. Now,
  if~$\crit_{i-1}(\Delta)=0$, then~$\im L_{i-1}=\ker\partial_{i-1}$,
  and~\ref{p1}$^\prime$ holds.  Conversely, suppose~\ref{p1}$^\prime$ holds, and
  let~$\sigma\in\ker\partial_{i-1}$.  By~\ref{p1}$^\prime$, there
  exist~$\tau\in\mathcal{E}$ and~$\phi\in\im
  L_{i-1}\subseteq\ker\partial_{i-1}$ such that~$\sigma=\tau+\phi$.  But
  then~$\sigma-\phi\in\mathcal{E}\cap\ker\partial_{i-1}=\left\{0 \right\}$,
  which implies~$\sigma=\phi\in\im L_{i-1}$.  It follows
  that~$\crit_{i-1}(\Delta)=0$.  Therefore,~\ref{p1} is equivalent to~\ref{p2}.

  We now prove the equivalence of~\ref{p2} and~\ref{p3} using
  Proposition~\ref{prop: forest number 1}.  If~$\crit_{i-1}(\Delta)=0$,
  then $1=|\T(\crit_{i-1})|=\tau_i(\Delta)$ by Theorem~\ref{thm: cut flow
  sequences}.  Further, the natural
  surjection~$\crit_{i-1}(\Delta)\to\tH_{i-1}(\Delta)$
  implies~$\tH_{i-1}(\Delta)=0$.  Hence, $\Skel_i(\Delta)$ is a
  spanning~$i$-tree of~$\Delta$.  Conversely, suppose that~$\Skel_i(\Delta)$ is
  a spanning~$i$-tree and~$\tH_{i-1}(\Delta)=0$.
  Then~$\tau_i(\Skel_i(\Delta))=1$, which implies that~$\crit_{i-1}(\Delta)$ is
  free by Theorem~\ref{thm: cut flow sequences}.  However, the free part
  of~$\crit_{i-1}(\Delta)$ is the same as the free part of~$\tH_{i-1}(\Delta)$
  by Corollary~\ref{cor: critical group and homology}.
  Therefore,~$\crit_{i-1}(\Delta)=0$.
\end{proof}


\begin{example}\label{example: staco}  This example shows that condition
  $\tH_{i-1}(\Delta)=0$ in part~\ref{p3} of Proposition~\ref{prop: staco} is
  necessary.  Consider the simplicial complex~$\Delta$ pictured in
  Figure~\ref{fig: staco}.  By inspection, $\tH_2(\Delta)=0$ and
  $\tH_{1}(\Delta)\simeq\Z\neq0$. So the complex is a forest but not a tree.

  One may compute directly that~$\crit_1(\Delta)\simeq \Z$ or argue as follows.
  By Proposition~\ref{prop: forest number 1}, we have~$\tau_2(\Delta)=1$.  By
  Theorem~\ref{thm: cut flow sequences}, it follows
  that~$|T(\crit_1(\Delta))|=1$.  Then Corollary~\ref{cor: critical group and
  homology} says~$\crit_1(\Delta)=\tH_1(\Delta)\simeq\Z$.  
  
  Now consider a generator for the first homology such as 
  \[
    \sigma=(0,0,0,1,-1,1)=\overline{23}-\overline{24}+\overline{34}.
  \]
  The Hilbert basis~$\Hilb_1$ for~$\ker^{+}L_1$, computed by Sage~(\cite{Sage}), is given by the rows of
  the table
  \[
  \begin{array}{cccccc}
      \overbar{12}&\overbar{13}&\overbar{14}&\overbar{23}&\overbar{24}&\overbar{34}\\\hline
      0 & 0 & 0 & 0 & 0 & 1 \\
      0 & 0 & 1 & 0 & 1 & 0 \\
      0 & 1 & 0 & 1 & 0 & 0 \\
      1 & 1 & 1 & 0 & 0 & 0
  \end{array}\ .
  \]
  Ordering the elements of~$\Hilb_1$ as they appear in the table, top-to-bottom,
  we have~$\deg(\sigma)=(1,-1,1,0)\not\geq0$. So~$\sigma$ is not winnable
  even though~$\partial_1(\sigma)=0\in\Lambda_0(\Delta)$.
\end{example}
\begin{figure}[ht] 
  \centering
  \begin{tikzpicture}[scale=0.4] 
  \draw (0,0) node[below]{$1$} -- (0,2) node[above]{$2$};
  \draw (-2.5,3.5) node[left]{$3$} -- (2.5,3.5) node[right]{$4$};
  \draw[fill=gray!30] (0,0)--(0,2)--(-2.5,3.5)--(0,0);
  \draw[fill=gray!30] (0,0)--(0,2)--(2.5,3.5)--(0,0);
\end{tikzpicture} 
\caption{A simplicial complex with facets~$\overbar{123}$, $\overbar{124}$, and
$\overbar{34}$ (cf.~Example~\ref{example: staco}).}\label{fig:
staco} 
\end{figure}
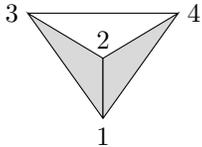

\section{Further work}\label{sect: further work}

There is still much to be learned about winnability of the dollar game on a
simplicial complex.  Here, we will present three general open areas of
investigation: computation of minimal winning degrees, algorithms for
determining winnability, and generalization of the rank function.

Theorem~\ref{thm: main} says there exists a realizable degree~$\delta$ such that
all~$i$-chains of degree at least~$\delta$ are winnable.  Call any minimal
such~$\delta$ a {\em minimal winning degree} for~$i$-chains on~$\Delta$.  For
divisors on connected graphs, there is one minimal winning
degree,~$g=|E|-|V|+1$.  We know of no such formulas in higher dimensions.  

\begin{enumerate}[leftmargin=*]
  \item Is there a simple combinatorial description of the
  set of minimal winning degrees for the~$i$-chains of a simplicial complex?  
  \item It would be nice to compute minimal winning degrees for a class of
    simplicial complexes.  For example,  what are the minimal winning degrees
    for~$(d-2)$-chains on the~$d$-dimensional simplex?  
\end{enumerate}

On a graph, there are three standard methods of determining whether the
dollar game is winnable, and if it is winnable, finding a sequence of moves leading to a
winning position.  One of these is a greedy algorithm.  It proceeds as follows:
\begin{enumerate}[leftmargin=*,label=(\roman*)]
  \item\label{gitem1} Check if the divisor is effective. If so, the divisor is
    winnable.  
  \item Modify the divisor by borrowing at any vertex with a negative amount of
    dollars, prioritizing vertices that have borrowed earlier in the algorithm.
  \item If all vertices have been forced to borrow, the original divisor is
    unwinnable.  Otherwise, return to step~\ref{gitem1}.
\end{enumerate}
The proof of the validity of this greedy algorithm
(cf.~\cite[Section~3.1]{Corry}) relies on two main facts. First, a vertex
cannot be brought out of debt by only borrowing at other vertices, and second,
the only way to leave a divisor unchanged through a series of
borrowing moves is to borrow at every vertex an equal number of times. Neither
of these two facts remains true for chains on a simplicial complex, so an immediate
translation of the greedy algorithm fails in higher dimensions. The ideas in
this paper suggest possible fixes for the second fact. For instance, one might
attempt to modify the algorithm to avoid borrowing at any combination of
vertices forming an element of the Hilbert basis~$\Hilb_i(\Delta)$ of the
nonnegative kernel~$\ker_i^+L_i$. Our attempts in this direction have failed due
to the first fact.  So we propose the question: 
\begin{enumerate}[leftmargin=*,start=3]
  \item  Can the greedy algorithm for the dollar game on graphs be generalized to
    one for simplicial complexes?
\end{enumerate}

Another method for determining winnability of the dollar game on a graph is
through~$q$-reduction of a divisor (\cite{Baker}, \cite{Shokrieh}).  In this
method, given a divisor, one computes a linearly equivalent standard form for
the divisor with respect to a chosen vertex~$q$.  The game is winnable if and
only if~$q$ is out of debt in this standard form.  Knowing whether~$q$-reduction
generalizes to chains on a simplicial complex would be of general interest to
the chip-firing community~(\cite[Problem 17]{AIMS}, \cite{Guzman2}).  Perhaps the
methods of~\cite{Perkinson} could be employed.  In that work,~$q$-reduction is
interpreted as an instance of Gr\"obner reduction of the lattice ideal of the
graph Laplacian.  We formulate the general question in the context of the dollar
game:
\begin{enumerate}[leftmargin=*,start=4]
  \item Can one define an efficiently computable standard representative of the
equivalence class of a chain on a simplicial complex which is effective if and
only if the chain is winnable?
\end{enumerate}

A third way of computing winnability for graphs is to determine whether a
certain simplex, defined using the columns of the Laplacian matrix, contains
integer points (cf.~\cite[Section 2.3]{Corry} or~\cite{Brauner}).  This method
easily extends to the dollar game on a simplicial complex, and it is the one we
use in our own computations.  However, the general problem of determining
whether a simplex has integer points is NP-hard unless the dimension is fixed.
Even so, for graphs,~$q$-reduction provides a method of determining winnability
of a divisor that is polynomial in the size of the divisor and the size of the
graph (\cite{Shokrieh}).
\begin{enumerate}[leftmargin=*,start=5]
  \item Is there any efficient algorithm for determining winnability of the
    dollar game on a simplicial complex?  
\end{enumerate}

The rank function, discussed in the introduction, is a measure of the robustness
of winnability of a divisor on a graph.  As noted in~\cite[Remark 1.13]{Baker},
for a divisor~$D$ on an algebraic curve, the same definition for rank would
give~$r(D)=\ell(D)-1$, where~$\ell(D)$ is the dimension of the vector space of
global sections of the line bundle associated with~$D$, appearing in the
standard formulation of the Riemann-Roch theorem for curves.  
The Riemann-Roch theorem for divisors~$D$ on an algebraic surface can be thought
of as a refinement of a lower bound on~$\ell(D)$ in terms of data associated
with~$D$ and the structure of the surface (by dropping the superabundance term).
This motivates the following:

\begin{enumerate}[leftmargin=*,start=6]
  \item Is there a generalization of the rank function to~$1$-chains on a
    simplicial complex of dimension~$2$, measuring robustness of
    winnability and perhaps related to the Riemann-Roch theorem for
    algebraic surfaces?  If so, can one find a combinatorial lower bound for it?
\end{enumerate}

\appendix\label{appendix}
\setcounter{secnumdepth}{0}  
\section{Appendix}
In this appendix, we prove Proposition~\ref{prop: pseudomanifold critical group}
and Theorem~\ref{thm: reduced Laplacian iso}.  The proof of Proposition~\ref{prop:
pseudomanifold critical group} requires the following lemma.

\begin{lemma}\label{lemma: pseudomanifold}
  Let $\Delta$ be a~$d$-dimensional orientable pseudomanifold without boundary. 
  Let~$\gamma_1,\dots,\gamma_m$ be the facets of~$\Delta$ oriented so
  that~$\gamma=\gamma_1+\dots+\gamma_m$ is a pseudomanifold orientation
  for~$\Delta$, i.e., such that~$\partial_d(\gamma)=0$.  Let $\sigma, \tau$ be
  two $(d-1)$-chains in the image of $\partial_d$, and write
\[
\sigma = \sum_{i=1}^m s_i\partial_d(\gamma_i) , 
\qquad \tau = \sum_{i=1}^mt_i\partial_d(\gamma_i) 
\] 
for some integers $\left\{ s_i \right\}$ and $\left\{ t_i \right\}$. Then
$\sigma$ and~$\tau$ are linearly equivalent if and only if $\sum_{i=1}^m
s_i=\sum_{i=1}^m t_i \bmod m$.
\end{lemma}
\begin{proof} Let $\xi$ be a $(d-1)$-face of $\Delta$.  Then~$\xi$ is contained
  in exactly two facets, say $\gamma_i$ and $\gamma_j$, and $L_{d-1}(\xi) =\pm(
  \partial_d(\gamma_i) - \partial_d(\gamma_j))$.  By strong connectivity, it
  follows that~$\partial_d(\gamma_i)-\partial_d(\gamma_j)$ is in the image
  of~$L_{d-1}$ for {\em
  any} pair~$1\leq i,j\leq m$, and thus,
  \[
    \im(L_{d-1})=\Span_{\Z}\left\{ \partial_d(\gamma_i)-\partial_d(\gamma_j):
    1\leq i,j\leq m \right\}
    =\textstyle\left\{ \sum_{i=1}^ma_i\partial_d(\gamma_i):\sum_{i=1}^ma_i=0 \right\}.
  \]
  So linear equivalence of~$\sigma$ and~$\tau$ is equivalent to being able to
  write
  \begin{equation}\label{eqn: cond1}
    \sum_{i=1}^{m}(s_i-t_i)\partial_d(\gamma_i)=\sum_{i=1}^{m}a_i\partial_d(\gamma_i)
  \end{equation}
  for some integers~$a_i$ summing to~$0$.  Since the~$\partial_d(\gamma_i)$ do
  not form a basis for the image of~$\partial_d$, we cannot directly conclude
  something about the relation between the coefficients on both sides of
  equation~\eqref{eqn: cond1}. However, note that the existence of arbitrary
  integers~$a_i$ (not necessarily summing to~$0$) such that equation~\eqref{eqn:
  cond1} holds is equivalent to
  \[
    \rho:=\sum_{i=1}^{m}(s_i-t_i-a_i)\gamma_i\in C_d(\Delta)
  \]
  being in~$\ker\partial_d=H_d(\Delta)=\Z\gamma$, and thus to the existence of
  an integer~$\ell$ such that~$\rho=\ell(\gamma_1+\dots+\gamma_m)$.
  In this case, since the~$\gamma_i$ form a basis for~$C_d(\Delta)$, we
  conclude~$s_i-t_i-a_i=\ell$ for~$i=1,\dots,m$.  Summing, we have
  \[
    \sum_{i=1}^{m}s_i=\sum_{i=1}^{m}t_i+\sum_{i=1}^{m}a_i\bmod m.
  \]
  The result follows: if~$\sigma$ and~$\tau$ are linearly equivalent, we can
  take~$\sum_{i=1}^{m}a_i=0$ and
  conclude that~$\sum_{i=1}^{m}s_i=\sum_{i=1}^{m}t_i\bmod m$.  Conversely,
  if~$\sum_{i=1}^{m}s_i=\sum_{i=1}^{m}t_i+\ell m$ for some integer~$\ell$,
  set~$a_i:=s_i-t_i-\ell$ for all~$i$.  Then~\eqref{eqn: cond1} holds, and
  so~$\sigma$ and $\tau$ are linearly equivalent.
\end{proof}

\begin{proof}[Proof of Proposition~\ref{prop: pseudomanifold critical group}]
The projection mapping from the critical group to the relative homology group in
codimension one gives the short exact sequence
\begin{equation}\label{eqn: pseudo}
  0 \to \im\partial_d/\im L_{d-1} \to \crit_{d-1}(\Delta) \to \tH_{d-1}(\Delta) \to 0.
\end{equation}
Let~$\gamma=\gamma_1+\dots+\gamma_m$ be as in the statement of
Lemma~\ref{lemma: pseudomanifold}, and 
first consider the case where~$\partial\Delta\neq\emptyset$.
Reasoning as in the beginning of the lemma, we still
have
\[
  X:=\Span_{\Z}\left\{\partial_d(\gamma_i)-\partial(\gamma_j):1\leq i,j\leq
  m\right\}\subseteq\im L_{d-1}.
\]
Given any~$f\in\partial\Delta$, there exists a unique~$\gamma_k$ whose
boundary contains~$f$ in its support.
Hence,~$L_{d-1}(f)=\pm\partial_d(\gamma_k)$.
Since~$\im L_{d-1}$ contains~$X$ and~$\partial_{d}(\gamma_k)$, it contains all
of the~$\im\partial_{d}(\gamma_i)$. So~$\im L_{d-1}=\im\partial_d$, and
hence,~$\crit_{d-1}(\Delta)=\tH_{d-1}(\Delta)$, as claimed.

Now consider the case where~$\partial\Delta=\emptyset$.  Since~$\Delta$ is an
orientable pseudomanifold,~$\tH_{d-1}(\Delta)$ is torsion-free, and thus
sequence~\eqref{eqn: pseudo} splits. By the lemma, the mapping
\begin{align*}
  \Z/m\Z&\to\im\partial_d/\im L_{d-1}\\
  k&\mapsto k\partial_d(\gamma_1)
\end{align*}
is an isomorphism.  The result follows.
\end{proof}

Our proof of Theorem~\ref{thm: reduced Laplacian iso} follows the general
outline of that in~\cite{DKM1} with substantial modifications.
\begin{proof}[Proof of Theorem~\ref{thm: reduced Laplacian iso}]  Considering the commutative diagram
\begin{center}
\begin{tikzcd}
  \Z\Upsilon_i\arrow[r,"\pup{i}"]\arrow[d]
  &\Z\Upsilon_{i-1}\arrow[r,"\pup{i-1}"]\arrow[d,equal]
  &\Z\Upsilon_{i-2}\arrow[d,equal]\\
  \Z\Delta_i\arrow[r,"\partial_{\Delta,i}"]
  &\Z\Delta_{i-1}\arrow[r,"\partial_{\Delta,i-1}"]
  &\Z\Delta_{i-2}
\end{tikzcd},
\end{center}
  we see
\[
  \im\pup{i}\subseteq\im\partial_{\Delta,i}\subseteq\ker\partial_{\Delta,i-1}=
  \ker\pup{i-1}.
\]
Thus, there is a short exact sequence
\[
  0\to\im\partial_{\Delta,i}/\im\pup{i}\to\tH_{i-1}(\Upsilon)\to\tH_{i-1}(\Delta)\to0.
\]
By hypothesis,~$\tH_{i-1}(\Upsilon)=\tH_{i-1}(\Delta)$, and hence
\begin{equation}\label{eqn: images equal}
  \im\pup{i}=\im\partial_{\Delta,i}.
\end{equation}

  We now describe a basis for~$\ker\partial_{\Delta,i}$.  For each~$\theta\in\Theta$,
  since~$\im\pup{i}=\im\partial_{\Delta,i}$, 
  \begin{equation}\label{eqn: coords}
    \partial_{\Delta,i}(\theta)=\sum_{\tau\in\Upsilon_i}a_{\theta}(\tau)\pup{i}(\tau)
  \end{equation}
  for some~$a_{\theta}(\tau)\in\Z$.  Since~$\tH_i(\Upsilon)=0$, the boundary
  mapping~$\pup{i}$ is injective, and thus the coefficients~$a_{\theta}(\tau)$
  are uniquely determined. Define
  \[
    \alpha(\theta):=\sum_{\tau\in\Upsilon_i}a_{\theta}(\tau)\tau
  \]
  and extend linearly to get a well-defined
  mapping~$\alpha:\Z\Theta\to\Z\Upsilon_i$.
  For each~$\theta\in\Theta$, let
  \[
    \hat{\theta}:=\theta-\alpha(\theta).
  \]
  We claim 
  \[
    \ker\partial_{\Delta,i}
    =\{\hat{\theta}:\theta\in\Theta\}.
  \]
  The~$\hat{\theta}$ are linearly independent elements of the kernel.
  To show they span, suppose~$\gamma=\sum_{\sigma\in\Delta_i}b_{\sigma}\sigma\in\ker\partial_{\Delta,i}$. 
  Consider
  \[
    \gamma':=\gamma-\sum_{\sigma\in\Theta}b_{\sigma}\hat{\sigma}
    =\sum_{\sigma\in\Upsilon_i}b_{\sigma}\sigma+\sum_{\sigma\in\Theta}b_{\sigma}(\sigma-\hat{\sigma})
    =\sum_{\sigma\in\Upsilon_i}b_{\sigma}\sigma+\sum_{\sigma\in\Theta}b_{\sigma}\alpha(\sigma).
  \]
  Then since~$\gamma$ and the~$\hat{\sigma}$ are in~$\ker\partial_{\Delta,i}$, so
is~$\gamma'$.  Further, since each~$\alpha(\sigma)\in\Z\Upsilon_i$, so
  is~$\gamma'$.  But~$\partial_{\Delta,i}$ restricted to~$\Upsilon_i$ is equal
  to~$\partial_{\Upsilon,i}$, which is injective. It follows that
  \[
    \gamma=\sum_{\sigma\in\Delta_i}b_{\sigma}\sigma=\sum_{\sigma\in\Theta}b_{\sigma}\hat{\sigma}.
  \]
  We thus have an isomorphism
  \[
    \pi\colon
    \Z\Theta\xrightarrow{\ \sim\ }\ker\partial_{\Delta,i}
  \]
  determined by~$\sigma\mapsto\hat{\sigma}$ with inverse given by setting
  elements of~$\Upsilon_i$ equal to~$0$:
  \[
    \sum_{\sigma\in\Delta_i}b_{\sigma}\sigma\xmapsto{\quad\ } \sum_{\sigma\in\Theta}b_{\sigma}\sigma.
  \]

  Next, we claim there is a commutative diagram with exact rows
  \begin{center}
    \begin{tikzcd}
      \Z\Theta\arrow[r,"\tilde{L}"]\arrow[d,"\iota"]
      &\Z\Theta\arrow[r]\arrow[d,"\pi","\sim" labl]
      &\cok\tilde{L}\arrow[r]\arrow[d,dashed]
      &0\\
      \Z\Delta_i\arrow[r,"L_i"]
      &\ker\partial_{\Delta,i}\arrow[r]
      &\crit_i(\Delta)\arrow[r]
      &0
    \end{tikzcd}
  \end{center}
  where~$\iota$ is the natural inclusion.  To check commutativity of the square
  on the left, let~$\theta\in\Theta$.  Then by definition of~$\tL$ and the fact
  that~$\iota(\theta)$ is supported on~$\Theta$,
  \[
    L_i\iota(\theta)=\rho+\tilde{L}\theta
  \]
  for some~$\rho\in\Z\Upsilon_i$.  We then
  have~$\pi^{-1}(\rho+\tilde{L}\theta)=\tilde{L}\theta$, as required.  Hence,
  there is a well-defined vertical mapping~$\cok\tL\to\crit_i(\Delta)$ on the
  right.  By the snake lemma, that mapping is an isomorphism if and only if
  the mapping
  \[
    \Z\Theta\to\Z\Delta_i/\ker L_i
  \]
  given by composing~$\iota$ with the quotient mapping is surjective.
  Therefore, to finish the proof, it suffices to show that for
  all~$\gamma\in\Upsilon_i$, there exists~$\delta\in\Z\Theta$ such
  that~$\gamma+\delta\in\ker L_i$ (so then~$\gamma=-\delta\bmod\ker L_i$).

  Now~$\ker L_i=\ker
  \partial_{\Delta,i+1}\partial_{\Delta,i+1}^t=\ker\partial_{\Delta,i+1}^t$. To get a
  description of~$\ker\partial_{\Delta,i+1}^t$, consider the exact sequence
  \[
    \Z\Delta_{i+1}\xrightarrow{\partial_{\Delta,i+1}}\Z\Delta_i\to\cok\partial_{\Delta,i+1}\to0.
  \]
  Applying the left-exact functor~$\Hom(\,\cdot\,,\Z)$, gives the exact sequence
  \begin{equation}\label{eqn: ker}
    \Z\Delta_{i+1}\xleftarrow{\partial_{\Delta,i+1}^t}\Z\Delta_i\leftarrow(\cok\partial_{\Delta,i+1})^*\leftarrow0,
  \end{equation}
  where we have identified~$\Z\Delta_{i}$ and~$\Z\Delta_{i+1}$ with their
  duals (using the bases~$\Delta_i$ and~$\Delta_{i+1}$, respectively).
  There is an exact sequence,
  \[
    0\to\ker\partial_{\Delta,i}/\im\partial_{\Delta,i+1}\to\Z\Delta_i/\im\partial_{\Delta,i+1}
    \to\Z\Delta_i/\ker\partial_{\Delta,i}\to0,
  \]
  i.e,
  \begin{equation}\label{eqn: cok}
    0\to\tH_i(\Delta)\to\cok\partial_{\Delta,i+1}\to\Z\Delta_i/\ker\partial_{\Delta,i}\to0.
  \end{equation}
  However,
  \[
    \Z\Delta_i/\ker\partial_{\Delta,i}\xrightarrow{\ \sim\
    }\im\partial_{\Delta,i}=\im\partial_{\Upsilon,i}\simeq\Z\Upsilon_i
  \]
  using~\eqref{eqn: images equal} and the fact that~$\partial_{\Upsilon,i}$ is
  injective.  Since~$\Z\Upsilon_i$ is free, sequence~\eqref{eqn: cok} splits:
  \begin{equation}\label{eqn: split}
    \cok\partial_{\Delta,i+1}\approx\tH_i(\Delta)\oplus\Z\Upsilon_i,
  \end{equation}
  with each~$\gamma\in\Upsilon_i$ identified with its class
  in~$\cok\partial_{\Delta,i+1}$.
  Given~$\gamma\in\Upsilon_i$, let~$\gamma^*\colon\Z\Upsilon_i\to\Z$ be the dual
  function.  Then use isomorphism~\eqref{eqn: split}, to identify~$\gamma^*$ with an element
  of~$(\cok\partial_{\Delta,i+1})^*$.  The
  image of~$\gamma^*$ in~$\Z\Delta_i$ under the mapping in~\eqref{eqn: ker} is
  \[
    \gamma+\sum_{\theta\in\Theta}a_{\theta}(\gamma)\theta,
  \]
  which by exactness of~$\eqref{eqn: ker}$ is an element
  of~$\ker\partial_{\Delta,i+1}^t$.
  Letting~$\delta:=\sum_{\theta\in\Theta}a_{\theta}(\gamma)\theta$, we
  see that~$\gamma+\delta\in\ker\partial_{\Delta,i+1}^t$, as required.
\end{proof}

\begin{remark} Theorem~\ref{thm: reduced Laplacian iso} generalizes Theorem~3.4
  of \cite{DKM1}.  Remark~3.5 of~\cite{DKM1} considers the case where~$\Delta$
  is the~$6$-vertex simplex,~$i=2$, and~$\Upsilon$ is a certain triangulation of
  the real projective plane (shown in Fig.~3 of~\cite{DKM3}).  In this case,
  \[
    \tH_1(\Delta)=0\neq\tH_1(\Upsilon)=\Z/2\Z,
  \]
  and
  \[
    \crit_2(\Delta)=\left(\Z/6\Z\right)^4\ \not\simeq\ \Z\Theta/\im\tL\simeq\left(
    \Z/12\Z \right)\oplus\left( \Z/6\Z
    \right)^3\oplus(\Z/2\Z).
  \]
  This example is given in~\cite{DKM1} to show that the
  condition~$\tH_{i-1}(\Delta)=\tH_{i-1}(\Upsilon)=0$ in Theorem~3.4 cannot be
  dropped.  Here, it serves the same purpose for the more relaxed
  hypothesis~$\tH_{i-1}(\Delta)=\tH_{i-1}(\Upsilon)$ of Theorem~\ref{thm:
  reduced Laplacian iso}.
\end{remark}

\bibliographystyle{amsplain}
\bibliography{simplicial_dollar_game.bib}

\providecommand{\bysame}{\leavevmode\hbox to3em{\hrulefill}\thinspace}
\providecommand{\MR}{\relax\ifhmode\unskip\space\fi MR }
\providecommand{\MRhref}[2]{%
  \href{http://www.ams.org/mathscinet-getitem?mr=#1}{#2}
}
\providecommand{\href}[2]{#2}
\begin{thebibliography}{10}

\bibitem{AIMS}
\emph{Problems from the {AIMS} {C}hip-{F}iring {W}orkshop},
  \url{https://aimath.org/WWN/chipfiring/aim_chip-firing_problems.pdf}, July
  2013.

\bibitem{Baker}
Matthew Baker and Serguei Norine, \emph{Riemann-{R}och and {A}bel-{J}acobi
  theory on a finite graph}, Adv. Math. \textbf{215} (2007), no.~2, 766--788.

\bibitem{Shokrieh}
Matthew Baker and Farbod Shokrieh, \emph{Chip-firing games, potential theory on
  graphs, and spanning trees}, J. Combin. Theory Ser. A \textbf{120} (2013),
  no.~1, 164--182.

\bibitem{Biggs}
N.~L. Biggs, \emph{Chip-firing and the critical group of a graph}, J. Algebraic
  Combin. \textbf{9} (1999), no.~1, 25--45.

\bibitem{Brauner}
Sarah Brauner, Forrest Glebe, and David Perkinson, \emph{Enumerating linear
  systems on graphs}, \url{https://arxiv.org/abs/1906.04768}, 2019.

\bibitem{CorryKeenan}
Scott Corry and Liam Keenan, private communication, 2017.

\bibitem{Corry}
Scott Corry and David Perkinson, \emph{Divisors and {S}andpiles}, American
  Mathematical Society, Providence, RI, 2018, An introduction to chip-firing.

\bibitem{DKM1}
Art~M. Duval, Caroline~J. Klivans, and Jeremy~L. Martin, \emph{Critical groups
  of simplicial complexes}, Ann. Comb. \textbf{17} (2013), no.~1, 53--70.

\bibitem{DKM2}
\bysame, \emph{Cuts and flows of cell complexes}, J. Algebraic Combin.
  \textbf{41} (2015), no.~4, 969--999.

\bibitem{DKM3}
\bysame, \emph{Simplicial and cellular trees}, Recent trends in combinatorics,
  IMA Vol. Math. Appl., vol. 159, Springer, [Cham], 2016, pp.~713--752.

\bibitem{Fulton}
William Fulton, \emph{Introduction to {T}oric {V}arieties}, Annals of
  Mathematics Studies, vol. 131, Princeton University Press, Princeton, NJ,
  1993, The William H. Roever Lectures in Geometry.

\bibitem{Guzman2}
Johnny Guzm\'{a}n and Caroline Klivans, \emph{Chip firing on general invertible
  matrices}, SIAM J. Discrete Math. \textbf{30} (2016), no.~2, 1115--1127.

\bibitem{Henk}
Martin Henk and Robert Weismantel, \emph{The height of minimal {H}ilbert
  bases}, Results Math. \textbf{32} (1997), no.~3-4, 298--303.

\bibitem{Hilbert}
David Hilbert, \emph{{\"U}ber die {T}heorie der algebraischen {F}ormen}, Math.
  Ann. \textbf{36} (1890), no.~4, 473--534.

\bibitem{Klivans}
Caroline~J. Klivans, \emph{The {M}athematics of {C}hip-{F}iring}, Discrete
  Mathematics and its Applications (Boca Raton), CRC Press, Boca Raton, FL,
  2019.

\bibitem{Massey}
William~S. Massey, \emph{A {B}asic {C}ourse in {A}lgebraic {T}opology},
  Graduate Texts in Mathematics, vol. 127, Springer-Verlag, New York, 1991.

\bibitem{Perkinson}
David Perkinson, Jacob Perlman, and John Wilmes, \emph{Primer for the algebraic
  geometry of sandpiles}, Tropical and non-{A}rchimedean geometry, Contemp.
  Math., vol. 605, Amer. Math. Soc., Providence, RI, 2013, pp.~211--256.

\bibitem{Schrijver}
Alexander Schrijver, \emph{Theory of {L}inear and {I}nteger {P}rogramming},
  Wiley-Interscience Series in Discrete Mathematics, John Wiley \& Sons, Ltd.,
  Chichester, 1986, A Wiley-Interscience Publication.

\bibitem{Spanier}
Edwin~H. Spanier, \emph{Algebraic {T}opology}, Springer-Verlag, New
  York-Berlin, 1981, Corrected reprint.

\bibitem{Sage}
{The Sage Developers}, \emph{{S}agemath, the {S}age {M}athematics {S}oftware
  {S}ystem ({V}ersion 8.2)}, 2018,
  \href{http://www.sagemath.org}{http://www.sagemath.org}.

\end{thebibliography}
\end{document}